%% file: Main.tex
\RequirePackage{etex}

\documentclass[10pt,psamsfonts]{amsart}
\usepackage{amsfonts}
\usepackage{amsmath}
\usepackage{amsthm, amssymb,  mathrsfs, epsfig}
\usepackage{color}
\usepackage{tikz}
\usepackage{tensor}
\usepackage{leftidx}
\usepackage[cmtip,all]{xy}
\usepackage{pb-diagram,pb-xy}
\usepackage{extarrows}

\DeclareMathOperator{\Hom}{Hom}

\usepackage[pagewise]{lineno}
\begin{document}

\newcommand{\Z}{\mathbb{Z}}
\newcommand{\So}{{\mathcal{S} }}
\newcommand{\M}{\mathcal{M}}
\newcommand{\N}{\mathcal{N}}
\newcommand{\D}{\mathcal{D}}
\newcommand{\F}{\mathcal{F}}
\newcommand{\ca}{{\mathcal{C} }}
\newcommand{\id}{\mbox{\rm id\,}}
\newcommand{\Aut}{\mbox{\rm Aut\,}}
\newcommand\vect{\operatorname{Vec}}
\newcommand\Inv{\operatorname{Inv}}
\newcommand\Rep{\operatorname{Rep}}
\newcommand{\Sym}{\text{S}}
\newcommand{\A}{\mathcal{A}}
\newcommand{\Q}{\mathcal{Q}}
\def\Ext#1#2{{#1}_{\tiny{#2}}^{ext}}
\def\uuPic{\underline{\underline{\text{Pic}}}}
\def\uuG#1{\underline{\underline{#1}}}
\def\lsup#1#2{\tensor[^#1]{#2}{}}
\def\lsuprsub#1#2#3{\tensor[^#1]{#2}{_#3}}

%%%ZW Definitions
\newcommand{\bcg}{\underline{\text{Aut}_{\otimes}^{br}(\mathcal{C})}}
\newcommand{\bbg}{\underline{\text{Aut}_{\otimes}^{br}(\mathcal{B})}}
%%%

%%% XC definitions. See illustrations in Sec 1.
\newcommand{\B}{\mathcal{B}} % braided category
\newcommand{\C}{\mathcal{C}} % fusion category
\def\Aut#1{\text{Aut}_{\otimes}^{br}(#1)} 
\def\uAut#1{\underline{\text{Aut}_{\otimes}^{br}(#1)}}
\def\uuAut#1{\underline{\underline{\text{Aut}_{\otimes}^{br}(#1)}}}
\def\cross#1#2{{#1}_{#2}^{\times}}
\def\gauge#1#2#3{{#1}_{#2}^{\times,#3}}
\def\Pic#1{\text{Pic}(#1)}
\def\uPic#1{\underline{\text{Pic}(#1)}}
\def\uuPic#1{\underline{\underline{\text{Pic}(#1)}}}
\def\u#1{\underline{#1}}
\def\uu#1{\underline{\underline{#1}}}

\def\SO{\text{SO}}
\def\SU{\text{SU}}
\def\unit{\mathbf{1}}
\def\extunit{\mathbf{1}}
\def\Fib{\text{Fib}}

\theoremstyle{definition}
\newtheorem{axiom}{Axiom}
\newtheorem{remark}{Remark}
\newtheorem{theorem}{Theorem}
\newtheorem{thm}{Theorem}
\newtheorem{Conjecture}{Conjecture}
\newtheorem{lem}{Lemma}
\newtheorem{example}{Example}
\newtheorem{cor}{Corollary}
\newtheorem{prop}{Proposition}
\newtheorem{rem}{Remark}
\newtheorem{definition}{Definition}
\newtheorem{measurement}{Measurement}
\newtheorem{ancilla}{Ancilla}

\numberwithin{equation}{section} \makeatletter
\renewenvironment{proof}[1][\proofname]{\par
    \pushQED{\qed}%
    \normalfont \topsep6\p@\@plus6\p@ \labelsep1em\relax
    \trivlist
    \item[\hskip\labelsep\indent
        \bfseries #1]\ignorespaces
}{%
    \popQED\endtrivlist\@endpefalse
} \makeatother
\renewcommand{\proofname}{Proof}

%%%%%%%%%%%
%=============================
% Comments
%=============================
\newcounter{commentcounter}

\newcommand{\commentc}[1]
{\stepcounter{commentcounter}
  \textbf{Comment \arabic{commentcounter} (by Cesar)}: 
{\textcolor{blue}{#1}} }

\newcommand{\commentj}[1]
{\stepcounter{commentcounter}
  \textbf{Comment \arabic{commentcounter} (by Julia)}: 
{\textcolor{pink}{#1}} }

\newcommand{\commentz}[1]
{\stepcounter{commentcounter}
  \textbf{Comment \arabic{commentcounter} (by Zhenghan)}: 
{\textcolor{red}{#1}} }

\newcommand{\comments}[1]
{\stepcounter{commentcounter}
  \textbf{Comment \arabic{commentcounter} (by Shawn)}: 
{\textcolor{orange}{#1}} }

%====================================

%%%%%%%%% TITULO %%%%%%%%%%%%%%%%%%%%%%%%%%%%%%%%%%
\title{On Gauging Symmetry of Modular Categories}

\author{Shawn X. Cui$^{1}$, C\'esar Galindo$^{2}$, Julia Yael Plavnik$^{3}$, and Zhenghan Wang$^{1,4}$}

\address{$^1$Department of Mathematics\\University of California\\Santa Barbara, CA 93106\\USA}
\email{xingshan@math.ucsb.edu, zhenghwa@math.ucsb.edu}
\address{$^2$Departamento de Matem\'aticas\\Universidad de los Andes\\Carrera 1 N. 18A -10, Bogot\'a\\
Colombia}
\email{cn.galindo1116@uniandes.edu.co}
\address{$^3$Department of Mathematics,
    Texas A\&M University,
    College Station, TX
    U.S.A.}
\email{julia@math.tamu.edu}
\address{$^4$Microsoft Research, Station Q\\ University of California\\ Santa Barbara, CA 93106\\USA}
\email{zhenghwa@microsoft.com}

\thanks{S.-X. C and Z.W. are partially supported by NSF DMS 1108736 and J.P. by CONICET, ANPCyT and Secyt-UNC. C. G. was partially supported by the FAPA funds from vicerrectoria de investigaciones de la Universidad de los Andes. This project began while J.P. was at Universidad de Buenos Aires, and the support of that institution is gratefully acknowledged. Part of this work was done during visits of C.G. to Microsoft Research Station Q and J.P. to University of California, Santa Barbara.}


\maketitle

\begin{abstract}
Topological order of a topological phase of matter in two spacial dimensions is encoded by a unitary modular (tensor) category (UMC). A group symmetry of the topological phase induces a group symmetry of its corresponding UMC. Gauging is a well-known theoretical tool to promote a global symmetry to a local gauge symmetry. We give a mathematical formulation of gauging in terms of higher category formalism. Roughly, given a UMC with a symmetry group $G$, gauging is a $2$-step process: first extend the UMC to a $G$-crossed braided fusion category and then take the equivariantization of the resulting category. Gauging can tell whether or not two enriched topological phases of matter are different, and also provides a way to construct new UMCs out of old ones. We derive a  formula for the $H^4$-obstruction, prove some properties of gauging, and carry out gauging for two concrete examples.
\end{abstract}

%%%%%%%%%%%%%%%%%%%%%%%%%%%%%%%%%%%%%%%%%%%%%%%%%%%%

\iffalse
%\frontmatter
\section{Notations}
\commentz{we should keep those notations somewhere: at the end of intro or the beginning of sec. 2 or here?}

$\C:$ fusion category

$\B:$ braided fusion category

$\Aut{\B}$, $\uAut{\B}$, $\uuAut{\B}$: The 1-group, 2-group, 3-group of braided equivalences of $\B$. 

Similarly, $\Pic{\B}, \uPic{\B}, \uuPic{\B}$: The 1-group, 2-group, 3-group of invertible module categories of $\B$.

$\cross{\B}{G}$: $G$-crossed braided category extended over $\B$

$\cross{\B}{(\rho,\alpha,\beta)}$: $G$-crossed braided category extended over $\B$

$\gauge{\B}{G}{G}, \gauge{\B}{(\rho,\alpha,\beta)}{G}$: the gauging of $\B$.

$\u{\rho}$: categorical group homomorphism

$\uu{\rho}$: categorical 2-group homomorphism

\commentc{Comment}
\comments{Comment}
\commentz{Comment}
\commentj{Comment}
%\tableofcontents
\fi
%%%%%%%%%%%%% INTRODUCCION %%%%%%%%%%%%%%%%%%%%%%%%

\input{Introduction}

%%%%%%%%%%%%%%%%%%%%%%%%%%%%%%%%%%%%%%%%%%%%%%%%%%%

%\mainmatter

%%%%%%%%%% CAPITULO 1 %%%%%%%%%%%%%%%%%%%%%%%%%%%%%
\input{preliminaries}

%%%%%%%%%%%%%%%%%%%%%%%%%%%%%%%%%%%%%%%%%%%%%%%%%%%%

%%%%%%%%%%% CAPITULO 2 %%%%%%%%%%%%%%%%%%%%%%%%%%%%%
\input{gaugingsymm}

%%%%%%%%%%%%%%%%%%%%%%%%%%%%%%%%%%%%%%%%%%%%%%%%%%%%

%%%%%%%%%%% CAPITULO 3 %%%%%%%%%%%%%%%%%%%%%%%%%%%%%

\input{sequentialgauging}

%%%%%%%%%%%%%%%%%%%%%%%%%%%%%%%%%%%%%%%%%%%%%%%%%%%%

%%%%%%%%%%% CAPITULO 4 %%%%%%%%%%%%%%%%%%%%%%%%%%%%%

\input{obstruction}
%%%%%%%%%%%%%%%%%%%%%%%%%%%%%%%%%%%%%%%%%%%%%%%%%%%%

\input{examples}

%%%%%%%%%%% CAPITULO 5 %%%%%%%%%%%%%%%%%%%%%%%%%%%%%
%\input{tesisjulidoc5}

%%%%%%%%%%%%%%%%%%%%%%%%%%%%%%%%%%%%%%%%%%%%%%%%%%%%

%%%%%%%%%%% CAPITULO 6 %%%%%%%%%%%%%%%%%%%%%%%%%%%%%
\input{appendix}

%%%%%%%%%%%%%%%%%%%%%%%%%%%%%%%%%%%%%%%%%%%%%%%%%%%%

%
%\renewcommand{\tablename}{Tabla}
%\renewcommand{\thesection}{A.\arabic{section}}
%\renewcommand{\theequation}{A.\arabic{equation}}
%\renewcommand{\thetable}{A.\arabic{table}}

%%%%%%%%%%%% BIBLIOGRAFIA %%%%%%%%%%%%%%%%%%%%%%%%%%
%\bibliographystyle{amsplain}
%\bibliographystyle{amsalpha}
%\bibliography{tesisdoc}

%\backmatter

%\bibliographystyle{amsbeta}
%
\input{bibliogauging}

%
%\newpage
%
%\printindex

\end{document}

%% file: Introduction.tex
\section{Introduction}\label{Introduction}

Topological phases of matter are quantum phases of matter represented by equivalence classes of gapped Hamiltonians.  In two spatial dimensions, the bulk topological order of a topological phase of matter $\mathcal{H}$ is encoded by a unitary modular (tensor) category (UMC) $\mathcal{B}$, also known as an anyon model \cite{ZWbook}.  Conventional symmetries of a topological phase $\mathcal{H}$ with topological order $\mathcal{B}$ induce topological symmetries of the UMC $\mathcal{B}$.  When a finite group $G$ acts on a topological phase $\mathcal{H}$ as topological symmetries, then gauging this global symmetry, when possible, leads to a topological phase transition from $\mathcal{H}$ to a new topological phase $\mathcal{H_{\text{gauged}}}$, whose topological order is encoded by a new UMC $\gauge{\B}{G}{G}$.  A physical theory of gauging based on $G$-crossed braided fusion category is developed in \cite{BBCW}. 

One reason for the interest in gauging comes from the study of symmetry enriched topological phases of matter (SETs).  Gauging can tell whether or not two SETs are different.
Another motivation is the classification of modular categories, which is interesting for both mathematics and condensed matter physics.  For ranks up to $5$, all modular categories are closely related to those that can be constructed from quantum groups \cite{BNRW}.  There are well-known constructions in conformal field theory that have analogues for modular category.  Gauging is another construction through which we can obtain new modular categories from old ones with group actions.  For example, all group-theoretical modular categories can be obtained from gauging a global symmetry of a pointed modular category by Prop. \ref{GTPMC} in Sec. 5.

Gauging is a well-known procedure in physics to promote a global symmetry to a local gauge symmetry.  While widely practiced in physics, gauging is hard to define mathematically.  In this paper, we formulate gauging with higher category formalism.  Our definition of gauging is the inverse of the so-called taking a core of a Tannakian subcategory in a modular category, which is called condensation of anyons in physics \cite{DGNO2}.  Our conceptual contribution is a formulation of gauging for two dimensional topological order modeled by an anyon model. Our definition is justified physically and leads to a study of the interplay of group symmetry and topological order based on three intertwined themes: symmetry fractionalization, defects, and gauging \cite{BBCW}.  Our technical new results include a formula for the notoriously hard to compute $H^4$-obstruction in Prop. \ref{H4formula} in Sec. 5, and a sequentially gauging lemma.  As an example, we carry out gauging for the first non-abelian symmetry $S_3$ of the UMC $SO(8)_1$ and obtain UMCs that have not appeared elsewhere.  In the earlier version, we obtained two UMCs that seemed to be two different UMCs with the same $T$-matrix. A referee pointed out that the two UMCs are equivalent.

The full symmetry of a set $X$ with $n$ identical elements is the permutation group $S_n$.  A group $G$ is a symmetry of $X$ if there exists a group homomorphism $\rho: G \rightarrow S_n$.  The full global symmetry group of a modular category $\mathcal{B}$ is the group $\Aut{\B}$ consisting of equivalence classes of braided tensor auto-equivalences of $\mathcal{B}$.  A group $G$ is a global symmetry of $\mathcal{B}$ if there exists a group homomorphism $\rho: G\rightarrow \Aut{\B}$.  Given a global symmetry of $\mathcal{B}$ for a finite group $G$, then symmetry defects can be introduced into the topological phase of matter.  Symmetry defects are modelled mathematically by simple objects in invertible module categories over $\mathcal{B}$.  The fundamental isomorphism $\Theta_{\mathcal{B}}: \uPic{\B}\rightarrow \uAut{\B}$ establishes a one-one correspondence between symmetries and defects \cite[Theorem 5.2]{ENO3}, where $\uPic{\B}$ is the $1$-truncation of the Picard $3$-group $\uuPic{\B}$ of invertible module categories.  A relation between defects and symmetries was studied earlier in \cite{FFRS,FPSV}.

Given a global symmetry $\rho: G \rightarrow \Aut{\B}$ for a finite group $G$, the first step in gauging $\rho$ is to add defects to $\mathcal{B}$ in a consistent way.  But there is an obstruction for introducing defects so that they form a fusion category.  The first  obstruction $O_3(\rho)$ is a cohomology class in $H^3(G;\text{Inv}(\mathcal{B}))$, where $\text{Inv}(\mathcal{B})$ is the group of invertible objects of $\mathcal{B}$.  In higher category formalism, this $H^3$-obstruction is the same as the obstruction to lifting $\rho$ to a categorical group homomorphism $\u{\rho}:\u{G}\rightarrow \uAut{\B}$.  When $O_3(\rho)$ vanishes, then we can define consistent fusion rules for defects, but the fusion rules are not unique.  The possible fusion rules are parameterized by cohomology classes $\alpha \in H^2(G;\text{Inv}(\mathcal{B}))$. For a given fusion rule specified by $(\rho, \alpha)$, the tensor product might not be associative, which leads to a secondary obstruction $O_4(\rho,\alpha)\in H^4(G;U(1))$.  This $H^4$-obstruction is the same as the obstruction to lifting the $2$-homomorphism $\u{\rho}:\u{G}\rightarrow \uPic{\B}$ to a tri-homomorphism $\uu{\rho}:\uu{G}\rightarrow \uuPic{\B}$ when $\uAut{\B}$ is identified with $\uPic{\B}$ by $\Theta_{\mathcal{B}}$.  When the $O_4(\rho,\alpha)$ obstruction vanishes, then we have consistent fusion categories for the defects, which are parameterized by cohomology classes $\beta\in H^3(G;U(1))$.  It follows that the resulting fusion category $\cross{\B}{G}$ is $G$-crossed braided with a categorical action of $G$.  The second step in gauging is to equivariantize the $G$-crossed braided fusion category $\cross{\B}{G}$.  There are no additional obstructions and the resulting modular category is called the gauged theory,  denoted as $\gauge{\B}{G}{G}$.  Most mathematical results needed for the above discussion are contained in \cite{ENO3}.

For both applications to physics and topology, it is important to compute the $H^4$-obstruction.  A formula for the obstruction $O_4(\rho,\alpha)$ in the quasi-trivial case is known \cite{Ga2}.  We reduce the computation of $O_4(\rho,\alpha)$ for the $G$-crossed extension $\cross{\B}{G}$ of $\B$ to the case of quasi-trivial extension of $\cross{\mathcal{Z}(\B)}{G}$.  The formula for the $H^4$-obstruction $\tilde{O}_4(\rho,\alpha)$ for the quasi-trivial $\cross{\mathcal{Z}(\B)}{G}$ case can be written in terms of the data of $\cross{\B}{G}$, and then our formula for $O_4(\rho,\alpha)$ in Prop. \ref{H4formula} follows.

The paper is organized as follows.  Sec. 2 contains preliminaries.  In Sec. 3, we define gauging and collect some general properties of gauging.  Sec. 4 is on a sequentially gauging lemma.  In Sec. 5, we derive formulas for both obstructions.  Finally, Sec. 6 contains two explicit examples.

%% file: preliminaries.tex
\section{Preliminaries}

In this section we recall some basic definitions and standard notions. Much of the material here can be found in \cite{DGNO2} and \cite{GHR}.  All our fusion categories are over the complex numbers $\mathbb{C}$. We will use the following notation in the paper.

$\C$, $\mathcal{D}$: fusion categories.

$\B$: a braided fusion category.

$\Aut{\B}$, $\uAut{\B}$: the 1-, 2-group of braided tensor auto-equivalences of $\B$. 

$\Pic{\B}, \uPic{\B}, \uuPic{\B}$: the 1-, 2-, 3-group of invertible module categories of $\B$.

$\cross{\B}{G}$: a $G$-crossed braided extension of $\B$.

$\cross{\B}{(\rho,\alpha,\beta)}$: a $G$-crossed modular extension of $\B$.

$\gauge{\B}{G}{G}, \gauge{\B}{(\rho,\alpha,\beta)}{G}$: the gauged $\B$.

$\u{\rho}$: $2$-group homomorphism.

$\uu{\rho}$: $3$-group homomorphism.

\subsection{Unitary fusion categories}\label{subsection unitary fusion}

A  \textbf{$C^*$-category} $\D$ is a $\mathbb C$-linear abelian
category with an involutive antilinear contravariant endofunctor $\dagger$,
which is the identity on objects. The hom-spaces $\Hom_\D(X,Y)$ are
Hilbert spaces with norms satisfying $$|| f g||\leq ||f||\ ||g||, \
\ ||f^\dagger f||=||f||^2,$$ for all $f \in \Hom_\D(X, Y), g \in\Hom_\D(Y,Z
)$, where  $f^\dagger$ denotes the image of $f$ under the endofunctor $\dagger$. A $C^*$-category is called locally finite dimensional if each hom-space $\Hom_\D(X,Y)$ is a finite dimensional Hilbert space. In this paper all $C^*$-categories will be locally finite dimensional.
\begin{rem}
A  $C^*$-structure over a locally finite dimensional complex abelian category $\D$ is the same as a positive complex $*$-structure, that is, an involutive antilinear contravariant endofunctor $\dagger$,
which is the identity on objects and such that for each $f \in \operatorname{Hom}_{\D}(X, Y )$, $f
^\dagger f = 0$ implies $f = 0$, see \cite[Proposition 2.1]{Mueg}. 
\end{rem}

Let $X$ and $Y$ be objects in a $C^*$-category. A morphism $u:X\to
Y$ is unitary if $uu^\dagger=\id_Y$ and  $u^\dagger u=\id_X$. A functor $F:\D\to \D'$ is called unitary is preserves the $*$-structure, that is, if $F(f^\dagger)= F(f)^\dagger$, for all $f\in \Hom_\D(X, Y)$. 

A \textbf{unitary fusion category} (UFC) is a fusion category $\C$, which
is a  $C^*$-category with all constraints unitary, and
$(f\otimes g)^\dagger= f^\dagger\otimes g^\dagger$ for every pair of morphisms $f,g$
in $\C$.

\subsection{Unitary braided fusion category and the center construction}

If $\C$ is a UFC, the center $\mathcal{Z}(\ca)$ is a \textit{unitary} braided fusion category (UBFC) and for all $(X,c_{-,X})\in \mathcal{Z}(\C)$, the natural isomorphisms $c_{W,X}:W\otimes X\to X\otimes W$ are unitary, for all $W\in \C$, \cite[Theorem 6.4]{Mueg2}. In particular, for every UFC $\C$, every braiding structure on $\C$ is unitary \cite[Theorem 3.2]{G}.

Since a UFC is always spherical, it follows that a UBFC is a unitary premodular category, in the sence of \cite{ENO}.

\subsection{Unitary modular categories and the symmetric center}

Two objects $X, Y\in \C$ in a braided fusion category centralize each other if $$c_{Y,X}c_{X,Y}=\id_{X\otimes Y}.$$
The symmetric center $\C'$ (or M\"{u}ger's center) is the full subcategory of $\C$ consisiting of objects that centralize each object of $\C$. The symmetric center is a symmetric fusion category. 

By \cite[Theorem 3.7]{DGNO2}, a UBFC is a unitary modular category (UMC) if and only if $\C'=\text{Hilb}$ (the category of finite dimensional Hilbert spaces).

\subsection{Unitary categorical actions and their equivariantizations}\label{categorical actions}

Let $\C$ be a UFC. We will denote by $\underline{\text{Aut}(\M)}$ (respectively, $\underline{\text{Aut}_\otimes(\C)}$) the monoidal category where objects are  unitary auto-equivalences of $\M$ (respectively, unitary tensor auto-equivalence of $\C$), arrows are  unitary natural isomorphisms (respectively, unitary tensor natural isomorphisms) and the tensor product is the composition of functors. 

A unitary  action of the group  $G$ on  $\C$ is a monoidal functor   $\u\rho:\underline{G}\to
\underline{\text{Aut}_\otimes(\C)}$.

Let $G$ be a group acting unitarily on $\C$ via $\u\rho:\underline{G}\to \underline{\text{Aut}_\otimes(\C)},$ then we have the following data \begin{itemize}
  \item unitary tensor functors $(\u\rho(g),\psi(g)): \C\to \C$ for each $g\in G$,
  \item unitary natural isomorphism $\phi(g,h): \u\rho(gh) \to \u\rho(g)\circ \u\rho(h)$ for all $g, h \in G$.
\end{itemize} 
The $G$-equivariantization (or category of $G$-invariant objects) of $\C$, denoted by $\C^G$, is a UFC defined as follows. An object in $\C^G$ is a pair $(V, f)$, where $V$ is an object of $\mathcal C$ and $f$ is a family of unitary isomorphisms $f_g: \u\rho(g)(V) \to V$, $g \in G$ such that for all $g, h \in G$, 
\begin{equation}\label{deltau} 
\phi(g,h)f_{gh}= f_g \circ \u\rho(g)(f_h).
\end{equation} A $G$-equivariant morphism $\phi: (V, f) \to (V', f')$ between $G$-equivariant objects $(V, f)$ and $(V', f')$ is a morphism $u: V \to V'$ in $\mathcal C$ such that $f'_g\circ \u\rho(g)(u) = u\circ f_g$, for all $g \in G$. The $C^*$-structure of $\C^G$ is the one inherited from $\C$. The tensor product is defined by \begin{align*}
    (V, f)\otimes (V', f'):= (V\otimes V', l),
\end{align*}where $$l_g= f_g f'_g\psi(g)_{V,V'}^*,$$and unit
object $(1, \text{id}_1)$.

%% file: gaugingsymm.tex
\section{Gauging a Global Symmetry}\label{gauging}

Gauging is an important theoretical tool in physics.  As an application to physics, we are interested in a mathematical formulation of gauging for symmetries of two dimensional topological phases of matter.  Mathematically, we consider gauging as a construction of new modular categories from old ones with group symmetry.  

For application to physics in our situation, all the discussion should be within the unitary setting.  However for the mathematical application and physics elsewhere, non-unitary is interesting too.  We will formulate the theory in the unitary setting, though most of the theory can be repeated in the non-unitary setting.  Throughout the paper, we need to use the basic notions in the unitary setting such as unitary Picard groups and the tensor product of unitary bi-module categories, which are defined in \cite{GHR}. In order to keep the notation simple, we continue to use the standard notation.

\subsection{Global symmetry of unitary modular categories}

A quantum system is modelled by a pair $(L,H)$, where $L$ is the (local) Hilbert space of states (or wave functions) and $H$ is the Hamiltonian---an Hermitian operator on $L$.  While we will not define the notion mathematically, we will refer to a class of gapped Hamiltonians without phase transitions among them as a topological phase of matter.  Elementary excitations in a two dimensional topological phase of matter form an anyon system, which is modelled by a UMC.  Therefore, we will say that the topological order of a two dimensional topological phase of matter is a UMC.  

A group $G$ is a symmetry of a quantum system $(L,H)$ if $G$ acts on $L$ unitarily and the action commutes with $H$, i.e., there is a group homomorphism $\rho: G\rightarrow \mathbb{U}(L)$ such that $\rho(g)H=H\rho(g)$ for all $g\in G$, where $\mathbb{U}(L)$ are the unitary operators of $L$.  When the quantum system $(L,H)$ represents a topological phase of matter whose topological order is given by a UMC $\mathcal{B}$, then the symmetry $(G,\rho)$ of $(L,H)$ induces a global symmetry of the UMC $\B$. Let $\Aut{\B}$ be the $1$-truncation of $\uAut{\B}$, i.e., the group of equivalence classes of braided tensor auto-equivalences of $\B$.

\begin{definition}

Given a group $G$ and a UMC $\B$, a global symmetry of $\B$ is a pair $(G,\rho)$, where $\rho: G\rightarrow \Aut{\B}$ is a group homomorphism.

\end{definition}

Given a UMC $\B$, it is in general difficult to compute $\Aut{\B}$ and $\uAut{\B}$.  One way to obtain interesting symmetries is to consider the $n$-fold Deligne product ${\B}^{\boxtimes n}$ of a UMC $\B$.  Then any subgroup $G$ of the permutation group $S_n$ is a global symmetry of ${\B}^{\boxtimes n}$.  Such obvious symmetries can also be combined with symmetries of $\B$.  For example, the full global symmetry group of $SO(16)_1\boxtimes SO(16)_1=SO(8)_1\boxtimes SO(8)_1$  contains at least $S_3\times S_3, \Z_2$.

\subsection{Symmetry defects}

While symmetries are intrinsic properties of a topological phase of matter, defects are extrinsic objects that are introduced to the topological phase of matter by modifying the Hamiltonian \cite{BBCW}.  For a topological phase of matter with topological order $\B$, we will model defects by simple objects in indecomposable module categories over $\B$.  We will refer to an indecomposable module category over $\B$ as a defect sector and if it is indexed by a group element $g$, we will refer to it as a flux sector with flux $g$.  Simple objects in a defect sector will be called defects.

Given a UFC $\C$, a left module category $\M$ over $\C$ is a $C^*$-category which is a categorical left representation of $\C$ compatible with the $C^*$-structure.  Similarly, we can define right module category and bi-module category over $\C$.  The tensor product $\boxtimes_{\C}$ of $\C$-bimodule categories was defined in \cite{ENO3}, see \cite{GHR} for definition of tensor product in the unitary setting. With this tensor product, a $(\C,\D)$-bimodule category $\M$ is called invertible if there is a $(\D,\C)$-bimodule $\N$ such that  $\M\boxtimes_{\D}\N \cong \C$ and $\N \boxtimes_{\C}\M \cong \D$ as bimodule categories.  The Brauer-Picard group $\text{BrPic}(\C)$ of $\C$ is the group of equivalence classes of invertible $\C$-bimodule categories. This group plays a key role in the classification of extensions of tensor categories by finite groups \cite[Theorem 1.3]{ENO3}.  The natural structure for invertible bi-module categories over a fusion category $\C$ is the $3$-group $\uu{\text{BrPic}(\C)}$, whose $1$-truncation is the $2$-group $\uu{\text{BrPic}(\C)}$. The Brauer-Picard group $\text{BrPic}(\C)$ of $\C$ is the $2$-truncation of $\uu{\text{BrPic}(\C)}$.

Note that for a braided fusion category $\B$, a left action induces a compatible right action via the braiding. In particular, all left $\B$-modules have a canonical $\B$-bimodule structure. It follows that in the braided case, there is a distinguished $3$-subgroup $\uu{\Pic\B} \subseteq \uu{\text{BrPic}(\B)}$ of the Brauer-Picard $3$-group, the so-called Picard $3$-group $\uu{\Pic\B}$ of $\B$ that consists of all invertible (left) $\B$-modules.  

\begin{definition}

Given a UMC $\B$, a symmetry defect of $\B$ is a simple object in an invertible module category over $\B$.

\end{definition}

Let $\uAut{\B}$ be the 2-group of braided unitary tensor auto-equivalences of a UBFC $\B$. There is a monoidal functor $\Theta:\u{\Pic\B}\to \uAut{\B}$ associated to the \emph{alpha-induction functors} $\alpha_{+}$ and $\alpha_{-}$, see \cite{ostrik1, ENO3} for precise definitions. When $\B$ is a UMC, there also exists a monoidal functor $\Phi: \uAut{\B}\to \u{\Pic\B}$ such that the functors $\Theta$ and $\Phi$ are mutually inverse equivalences of $\uAut{\B}$ and $\u{\Pic\B}$ as $2$-groups \cite[Theorem 5.2]{ENO3}.

\begin{definition}
A unitary (faithfully) $G$-crossed braided fusion category $\cross{\B}{G}$ is a unitary fusion category $\cross{\B}{G}$ equipped with the following structures:

\begin{itemize}
\item a unitary action of $G$ on $\cross{\B}{G}$;
\item a faithful $G$-grading $\cross{\B}{G}=\bigoplus_{\sigma\in G}\B_\sigma$;
\item unitary natural isomorphisms $$c_{X,Y}:X\otimes Y\to g(Y)\otimes X, \  \  g\in G, X\in \B_g,Y\in \cross{\B}{G},$$
this unitary morphisms are called the $G$-braiding.

This data should satisfy the following conditions:
\begin{itemize}
\item $g(\B_h)=\B_{ghg^{-1}}$, for all $g,h\in G$,
\item $g(c_{X,Y})=c_{g(X),g(Y)}$, for all $g\in G$,
\item and some commuting diagrams that guarantee the naturality of $c$, the consistency of $c$ with the tensor product, etc. See \cite[Definition 4.41]{DGNO2}.
\end{itemize}
\end{itemize}
\end{definition}

A unitary $G$-crossed braided fusion category $\cross{\B}{G}$  has an extended $S$-matrix (see \cite[Section 9]{Kiri}).  We say a unitary $G$-crossed braided fusion category is modular if its extended modular $S$-matrix is non-singular.  Verlinde formulas and modular representations can be generalized \cite{BBCW,Kiri}.

\begin{prop}\cite{DGNO2}

A unitary $G$-crossed braided extension $\cross{\B}{G}$ of a UMC $\B$ is modular if and only if $\B$ is modular.

\end{prop}

A $G$-crossed braided fusion category $\cross{\B}{G}$ decomposes as a direct sum $\cross{\B}{G} =\bigoplus_{g\in G} \B_g$, where each component is an abelian full subcategory of $\cross{\B}{G}$ and the tensor product maps $\B_g\times \B_h$ to $\B_{gh}$, i.e. $\cross{\B}{G}$ is a $G$-graded category, equipped with an action of $G$ compatible with the grading and a $G$-braiding. Notice that the trivial component $\B_e$ of the grading is a braided fusion category, each component $\B_g$ is an \emph{invertible} $\B_e$-module category and the functors $M_{g,h}:\B_g\boxtimes_{\B_e}\B_h\to \B_{gh}$ induced from the tensor product by restriction are $\B_e$-module equivalences, by \cite[Theorem 6.1]{ENO3}. 

A $G$-crossed braided fusion category $\cross{\B}{G} =\bigoplus_{g\in G} \B_g$ determines and is determined by the following data:

\begin{itemize}
  \item a BFC  $\B_e=\B$, a collection of invertible $\B$-module categories $\B_g, g\in G$,
  \item a collection of $\B$-module equivalences $M_{g,h}:\B_g\boxtimes_{\B_e} \B_h\to \B_{gh}$,
  \item natural isomorphisms of $\B$-module functors \[\alpha_{g,h,k}:M_{g,h k}(\text{Id}_{\B_g}\boxtimes_{\B_e} M_{h,k})\to M_{gh,k}(M_{g,h}\boxtimes_{\B_e} \text{Id}_{\B_k})\] satisfying certain identities.
\end{itemize}

We are interested in the opposite direction: when a given collection of defect sectors in $\uuPic{\B}$ would form a $G$-crossed braided extension of the UMC $\B$? It follows from \cite[Theorem 8.4, 8.8]{ENO3} that a (faithfully graded) $G$-crossed braided fusion extension of $\B$ exists if and only if a certain tensor product obstruction class in $H^3(G,\text{Inv}(\B))$ and a secondary associativity constraint obstruction class in $H^4(G,U(1))$ vanish. 

\subsection{Definition of gauging}

Given a global symmetry $(G,\rho)$ of a quantum system $(L,H)$, gauging in physics is to couple gauge fields to the Hamiltonian $H$ to promote the global symmetry $G$ to a local gauge symmetry.  There is neither a straightforward nor unique way to gauge.  The common practice in the Hamiltonian formalism is to choose the so-called minimal coupling by replacing ordinary derivatives with covariant derivatives.  The first step in gauging is to add flux sectors of defects into the theory.  For a topological order $\B$, we need to add defects to $\B$ to form a $G$-crossed modular extension $\cross{\B}{G}$ of $\B$. Such defects are in general confined, so in the second step we equivariantize the $G$-crossed extension $\cross{\B}{G}$, which leads to a new topological order  $\gauge{\B}{G}{G}$.  The first step has obstructions and ambiguities, while the second step has no further obstructions and is unique.  

The first step of adding defects consistently amounts to a lifting of the global symmetry $\rho: G\rightarrow \Aut{\B}$ to a $2$-homomorphism $\u{\rho}: \u{G}\rightarrow \u{\Aut{\B}}$ in the first stage and then a further lifting of $\u{\rho}: \u{G}\rightarrow \u{\Aut{\B}}$ to a tri-homomorphism $\uu{\rho}: \uu{G}\rightarrow \uuPic{\B}$ in the second stage when $\u{\Aut{\B}}$ and $\uPic{\B}$ identified by $\Theta_{\B}$.  This motivates the following definition:

\begin{definition}
Given a global symmetry  $\rho: G\rightarrow \Aut{\B}\cong \Pic{\B}$, $(G, \rho)$ can be gauged if there exists a tri-homomorphism $\uu{\rho}: \uu{G}\rightarrow \uu{\Pic \B}$ such that ${\rho}$ is equivalent to $\overline{\overline{\uu{\rho}}}$, where $\overline{\overline{(-)}}:\uu{\Pic \B}\to {\Pic \B}$ is the $2$-truncation map.
\end{definition}

The following theorem, \cite[Theorem 7.12]{ENO3}, will be used throughout the paper.
\begin{theorem}\cite{ENO3}\label{ENO3712}
$G$-crossed braided extensions $\cross{\B}{G}$ of $\B$ having a faithful $G$-grading are in bijection with tri-homomorphisms $\uu{\rho}: \uu{G}\rightarrow \uu{\Pic \B}$, or equivalently with homotopy classes of maps between their classifying spaces $ \text{BG} \to \text{B}\uuPic{B}$ 
\end{theorem}

%Note that a tri-homomorphism $\uu{\rho}: \uu{G}\rightarrow \uu{\Pic \B}$ is the same as having a braided G-crossed structure on $\B$, or equivalently a %map between their classifying spaces $ \text{BG} \to \text{B}\uuPic{B}$, by \cite[Theorem 7.12]{ENO3}.  Thus,

\begin{definition}

If a global symmetry $(G, \rho)$ of a UMC $\B$ can be gauged, then gauging is the two-step process that firstly $\B$ is extended to a unitary $G$-crossed braided fusion category $\cross{\B}{G}$, and secondly $\cross{\B}{G}$ is equivariantized to a UBFC $\gauge{\B}{G}{G}$.

\end{definition}

\begin{prop}\cite{Kiri}

A unitary $G$-crossed braided fusion extension $\cross{\B}{G}$ of a UMC $\B$ is modular if and only if $\gauge{\B}{G}{G}$ is modular.

\end{prop}

\subsection{Obstructions for gauging and gauging data}

By Theorem \ref{ENO3712}, a global symmetry can be gauged if the global symmetry can be lifted to a map between classifying spaces $\text{BG} \to \text{B}\uuPic{B}$.  Using homotopy theory, we see that such liftings can have obstructions. The first lifting from the global symmetry $\rho: G\rightarrow \Aut{\B}$ to a $2$-homomorphism $\u{\rho}: \u{G}\rightarrow \u{\Aut{\B}}$ is the promotion of a group action on $\B$ to a $2$-group action on $\B$ by monoidal functors.  We will call this categorical action of $G$ on $\B$ a topological symmetry, i.e., a topological symmetry of $\B$ is a pair $(\u{G}, \u{\rho})$ such that $\u{\rho}: \u{G}\rightarrow \uAut{\B}$.  Topological symmetry in our sense is different from the topological symmetry in \cite{BBCW}, where topological symmetry refers to the full global symmetry group $\Aut{\B}$.  A topological symmetry $\u{\rho}: \u{G}\rightarrow \uAut{\B}$ can be gauged if $\u{\rho}$ can be lifted to a $\uu{\rho}: \uu{G}\rightarrow \uu{\Pic \B}$ when $\uAut{\B}$ is identified with $\u{\Pic \B}$

The $2$-group $\u{\Aut{\B}}$ has as a complete invariant the triple $(\Aut{\B}, \operatorname{Aut}_{\otimes}(\operatorname{Id}_\B), \phi)$, where $\phi$ is a cohomology class in $H^3(\Aut{\B}, \operatorname{Aut}_{\otimes}(\operatorname{Id}_\B))$.  Then $\rho$ can be lifted if and only if the pull-back cohomology class $O_3(\rho)={\rho}^{*}(\phi)\in H^3(G,\operatorname{Aut}_{\otimes}(\operatorname{Id}_\B))$ vanishes.  We will call this obstruction $O_3(\rho)$ the $H^3$-obstruction of $\rho$.  If this $H^3$-obstruction vanishes, then the possible liftings are parametrized by classes $\alpha \in H^2(G,\text{Inv}(\B))$.  Suppose a lifting specified by $(\rho, \alpha)$ is given, then to lift $\u{\rho}: \u{G}\rightarrow \u{\Aut{\B}}$ to a tri-homomorphism $\uu{\rho}: \uu{G}\rightarrow \uu{\Pic \B}$ has a further obstruction  $O_4(\rho, \alpha)\in H^4(G,U(1))$.  We will call this secondary obstruction the $H^4$-obstruction.  If this $H^4$-obstruction vanishes, the possible liftings are parametrized by cohomology classes $\beta \in  H^3(G,U(1))$.

\begin{definition}

Given a global symmetry  $\rho: G\rightarrow \Aut{\B}$ that can be gauged and a fixed gauging, then a gauging data related to the fixed gauging is a pair $(\alpha, \beta)\in H^2(G,\text{Inv}(\B))\times  H^3(G,U(1))$,  such that $O_4(\rho, \alpha)$ vanishes.

\end{definition}

Phrasing Theorems 8.4, 8.8, 8.9 from \cite{ENO3} in the language of gauging, we have the following proposition.

\begin{prop}\cite{ENO3}
Suppose $\B$ is a UMC with a gauging data $(\rho,\alpha,\beta)$ and $\cross{\B}{(\rho,\alpha,\beta)}$ is the extension of the UMC $\B$ to a unitary $G$-crossed fusion category, then $\cross{\B}{(\rho,\alpha,\beta)}$ has a canonical $G$-braiding and categorical $G$-action that make it into a unitary $G$-crossed modular category. 
\end{prop}

Thus the gauging data is the information required to extend a UMC $\B$ to a unitary $G$-crossed modular category $\cross{\B}{G}$ uniquely, whose equivariantization is the gauged UMC  $\B_{gauged}=\gauge{\B}{G}{G}$.

\subsection{General properties}

\begin{prop}\cite{GNN}
Suppose $G$ acts categorically on $\B=\mathcal{Z}(\C)$---the Drinfeld center of a unitary fusion category $\C$. Let $\cross{\C}{G}$ be the $G$-crossed braided extension of $\C$ from the $G$-action, then
$\gauge{\B}{G}{G}=\mathcal{Z}(\cross{\C}{G})$.
\end{prop}

Recall a fusion category is \emph{weakly integral} if its Frobenius-Perron dimension is an integer.  Two modular categories $\B$ and $\widetilde{\B}$ are \emph{Witt equivalent} if there exist spherical fusion categories
$\C$ and $\widetilde{\C}$ such that $\B\boxtimes \mathcal{Z}(\C) \cong \widetilde{\B}\boxtimes \mathcal{Z}(\widetilde{\C})$ as braided tensor categories, where $\mathcal{Z}(\C)$ and $\mathcal{Z}(\widetilde{\C})$ are the Drinfeld centers of $\C$ and $\widetilde{\C}$, respectively.

The following theorem follows from \cite{DGNO2} on taking a core, which is the inverse of gauging, and Corollary 3.30 of \cite{DMNO}.

\begin{theorem}\cite{DGNO2,DMNO}
Let $\B$ be a UMC with a gauging data $(\rho, \alpha, \beta)$, and $\gauge{\B}{G}{G}$ be its gauged UMC. Then $\B \otimes \overline{\gauge{\B}{G}{G}}\cong \mathcal{Z}(\cross{\B}{(\rho,\alpha,\beta)})$.

It follows that
\begin{enumerate}
\item Gauging preserves topological central charge.
\item dim($\gauge{\B}{G}{G}$) $= |G|^2$ dim($\B$).
\item  $\B$ is weakly integral if and only if $\gauge{\B}{G}{G}$ is  weakly integral.  In particular, if $\B$ is pointed then $\gauge{\B}{G}{G}$ is weakly integral. 
\item Gauging preserves Witt-equivalence classes.
\end{enumerate}
\end{theorem}

%% file: sequentialgauging.tex
\section{Sequentially Gauging}

In this section, we show that if the global symmetry group $G$ of $\B$ has a semi-product structure, i.e. $G = N \rtimes H$, then the gauging process can be done sequentially, that is, one can first gauge $\B$ by the normal subgroup $N$, and then gauge the resulting $\gauge{\B}{N}{N}$  by $H$. We show that for any $\uu{\rho} : \uu{G} \longrightarrow \uuPic{\B}$, there exist $\uu{\rho_1}: \uu{N} \longrightarrow \uuPic{\B}, \, \uu{\rho_2} : \uu{H} \longrightarrow \uuPic{\gauge{\B}{N}{N}} $, such that $\gauge{\B}{G}{G}$ is equivalent to $\gauge{(\gauge{\B}{N}{N})}{H}{H}$. 

%\commentc{We must cite the remark of \cite{ENO3}, where they said it can be done for an arbitrary group and normal subgroup.}
\begin{rem}
In \cite{DGNO2}, the authors gave a similar statement as above, without a proof, that equivaritization can be done sequentially for a fusion category with a $G$-action.
\end{rem}

By Theorem \ref{ENO3712}, morphisms $\uu{\rho} : \uu{G} \longrightarrow \uuPic{\B}$ are in bijection with $G$-crossed braided extensions of $\B$, $\cross{\B}{G} = \bigoplus\limits_{g \in G} \B_g$. The action of $G$ on $\cross{\B}{G}$, denoted by $R_{\rho}$, is defined as follows. $R_{\rho}(g)$ is the equivalence $\B_{g'} \longrightarrow \B_{gg'g^{-1}}$, such that the following $\B$-module functors are isomorphic

\begin{equation} \label{G-crossed}
\bullet \otimes X \cong R_{\rho}(g)(X) \otimes \bullet : \B_g \longrightarrow \B_{gg'}, \quad \forall X \in \B_{g'}
\end{equation}

Extending it to $\cross{\B}{G}$ linearly, we get a $G$-action $R_{\rho}$. Moreover, the isomorphism in (\ref{G-crossed}) gives the $G$-crossed braiding
$$c_{Y,X}: Y \otimes X \overset{\sim}{\longrightarrow} R_{\rho}(g)(X) \otimes Y, \, \quad Y \in \B_g, X \in \B_{g'}.$$
See  \cite[Theorem 7.12]{ENO3} for a more detailed explanation.

%When no confusion arises, we also write $R_{\rho}(g)(\cdot)$ as %$\lsup{g}{{(\cdot)}}$.

For $h \in H$, let $\C_h = \bigoplus\limits_{n \in N} \B_{hn}$. Thus we have $\cross{\B}{G} = \bigoplus\limits_{h \in H} \C_h$. Let $\uu{\rho_1} = \uu{\rho}_{|_{N}}$ be the restriction of $\uu{\rho}$ on $\uu{N}$.

\begin{lem} \label{lem:rho1=rho}
Let $\uu{\rho}, \uu{\rho_1}$ be as above, then $\C_e = \bigoplus\limits_{n \in N} \B_{n} $ with the $N$-crossed braided structure induced from $\cross{\B}{G}$ is the $N$-extension of $\B$ corresponding to $\uu{\rho_{1}},$ namely $\C_e = \cross{\B}{N}.$
\begin{proof}
Apparently, $\C_e$ is an $N$-extension of $\B$ with $\uu{\rho_1}(n) = \B_n, n \in N$. We only need to show the action of $N$ on $\C_e$ induced from $R_{\rho}$ is the same as that determined by $\uu{\rho_1},$ namely, $R_{\rho}(n) = R_{\rho_1}(n)$, but this is clearly true since both actions are determined from (\ref{G-crossed}).
\end{proof}
\end{lem}

Now we restrict the action $R_{\rho}$ to the subgroup $N$. By Lemma \ref{lem:rho1=rho}, $R_{\rho}(n)(X) = R_{\rho_1}(n)(X), n \in N, X \in \C_e.$
If $X \in \B_{hn'}, h \in H$,  then $\lsup{n}{X} \in \B_{nhn'n^{-1}}$. Since $nhn'n^{-1} = h (h^{-1}nh)n'n^{-1} \in hN$, we have $\lsup{n}{X} \in \C_h$. Therefore, the action of $N$ preserves each $\C_h$.

Now we take the equivariantization of $\cross{\B}{G}$ with respect to the action of $N$. By the argument above, the equivariantization preserves each $\C_h$. Thus we have $\gauge{\B}{G}{N} = \bigoplus\limits_{h \in H} \C_h^{N}$, where $\C_e^{N} = \gauge{\B}{N}{N}$ by definition.

\begin{lem}\label{lem:gauge N}
$\gauge{\B}{G}{N}$ is an $H$-crossed braided category which is an $H$-extension of $\C_e^{N}= \gauge{\B}{N}{N},$ and thus we get a morphism $\uu{\rho_2}: \uu{H} \longrightarrow \uuPic{\gauge{\B}{N}{N}}$ corresponding to this extension.
\begin{proof}

Firstly, $\C_{h_1}^{N} \otimes \C_{h_2}^{N} \subset \C_{h_1h_2}^N$. This is direct to check. For $(X_i, \varphi_i) \in \C_{h_i}^N$, $X_i \in \B_{h_in_i}, i = 1, 2$, we have $(X_1, \varphi_1) \otimes (X_2, \varphi_2) = (X_1 \otimes X_2, \varphi_1\otimes \varphi_2).$ Then $X_1 \otimes X_2 \in \B_{h_1n_1h_2n_2} = \B_{h_1h_2(h_2^{-1}n_1h_2n_2)} \subset \C_{h_1h_2}^{N}$.

Secondly, we define an $H$-action on $\gauge{\B}{G}{N}$. Recall that for $(X, \varphi) \in \C_{h'}^{N},$ we have, for each $n \in N$, an isomorphism $\lsup{n}{X} \overset{\varphi_n}{\longrightarrow} X$. For any $h \in H$, we define $R(h)(X, \varphi) := (\lsup{h}{X}, \psi),$  where $\psi(n) =  \lsuprsub{h}{\varphi}{{h^{-1}nh}}: \lsup{{nh}}{X} \longrightarrow \lsup{h}{X} $. We also write $\psi = \lsuprsub{h}{\varphi}{{h^{-1} \bullet h}}$. For a morphism $f: (X, \varphi) \longrightarrow (X', \varphi'),$ define $R(h)(f) := \lsup{h}{f}.$ It is straightforward to check that $R$ is an action of $H$ on $\gauge{\B}{G}{N}$ by tensor automorphisms and $R(h)(\C_{h'}^{N}) \subset \C_{hh'h^{-1}}^{N}.$

The $H$-crossed braiding is given as follows: given $(X_{h_1n_1},\varphi) \in \C_{h_1}^N, (X_{h_2n_2},\varphi') \in \C_{h_2}^N,$ where $X_{h_in_i} \in \B_{h_in_i}, i =1,2,$ we have $(X_{h_1n_1},\varphi) \otimes (X_{h_2n_2},\varphi') = (X_{h_1n_1} \otimes X_{h_2n_2}, \varphi \otimes \varphi'),$ and $R(h_1)(X_{h_2n_2},\varphi') \otimes (X_{h_1n_1},\varphi) = (\lsuprsub{{h_1}}{X}{{h_2n_2}} \otimes X_{h_1n_1}, \lsuprsub{{h_1}}{{\varphi'}}{{h_1^{-1}\bullet h_1}} \otimes \varphi), $ then the crossed braiding is defined as the following compositions:
$$X_{h_1n_1} \otimes X_{h_2n_2} \overset{c}{\longrightarrow} \lsuprsub{{h_1n_1}}{X}{{h_2n_2}} \otimes X_{h_1n_1} \overset{\lsuprsub{{h_1}}{{\varphi'}}{{n_1}} \otimes Id}{\longrightarrow} \lsuprsub{{h_1}}{X}{{h_2n_2}} \otimes X_{h_1n_1},$$
where $c$ is the crossed braiding in $\cross{\B}{G}$. Again, it is not hard to check this defines an $H$-crossed braiding.

\end{proof}
\end{lem}
Therefore, by Lemma \ref{lem:gauge N}, we can take the equivariantization of $\gauge{\B}{G}{N}$ with respect to $H$, namely, $[\gauge{\B}{G}{N}]^{H} = \gauge{(\gauge{\B}{N}{N})}{H}{H}.$ The following theorem proves that $[\gauge{\B}{G}{N}]^{H} $ is braided equivalent to $\gauge{\B}{G}{G}$.

\begin{thm}\label{thm:seq gauge}
Let $\uu{\rho} : \uu{G} \longrightarrow \uuPic{\B}$, then there exist $\uu{\rho_1}: \uu{N} \longrightarrow \uuPic{\B}, \, \uu{\rho_2} : \uu{H} \longrightarrow \uuPic{\gauge{\B}{N}{N}} $, such that $\gauge{(\gauge{\B}{N}{N})}{H}{H}$ is braided equivalent to $\gauge{\B}{G}{G}$.
\begin{proof}
Let $\uu{\rho_1} = \uu{\rho}_{|_{N}},$ $\uu{\rho_2}$ be provided in Lemma \ref{lem:gauge N}.

We shall prove the theorem by defining a functor $F: \gauge{(\gauge{\B}{N}{N})}{H}{H} \longrightarrow \gauge{\B}{G}{G}$ and its inverse $F^{-1}$.

Given $((X,\varphi),\psi) \in \gauge{(\gauge{\B}{N}{N})}{H}{H},$ we have isomorphisms
\begin{equation}
\lsup{n}{X} \overset{\varphi_n}{\longrightarrow} X, \quad \forall n \in N
\end{equation}
\begin{equation}
(\lsup{h}{X}, \lsuprsub{h}{\varphi}{{h^{-1}\bullet h}}) \overset{\psi_{h}}{\longrightarrow} (X, \varphi), \quad \forall h \in H
\end{equation}
such that the following diagrams commute:

\begin{equation} \label{Equ: varphi}
\xymatrix{
\forall n_1, n_2 \in N &\lsup{{n_1n_2}}{X} \ar[r]^-{\lsuprsub{{n_1}}{\varphi}{{n_2}}} \ar[dr]_-{\varphi_{n_1n_2}} & \lsup{{n_1}}{X} \ar[d]^-{\varphi_{n_1}}\\
& & X\\
}
\end{equation}

\begin{equation} \label{Equ:varphi psi}
\xymatrix{
\forall n \in N, h \in H &  \lsup{{nh}}{X} \ar[r]^-{\lsuprsub{h}{\varphi}{{{h^{-1}nh}}}} \ar[d]^-{\lsuprsub{n}{\psi}{h}} & \quad \lsup{h}{X} \ar[d]^-{\psi_h} \\
 & \lsup{n}{X} \ar[r]^-{\varphi_n} & X \\
}
\end{equation}

\begin{equation} \label{equ:psi}
\xymatrix{
\forall h_1, h_2 \in H &\lsup{{h_1h_2}}{X} \ar[r]^-{\lsuprsub{{h_1}}{\psi}{{h_2}}} \ar[dr]_-{\psi_{h_1h_2}} & \lsup{{h_1}}{X} \ar[d]^-{\psi_{h_1}}\\
& & X\\
}
\end{equation}

Define $F((X, \varphi), \psi):= (X, \tau), $ where for any $g = nh \in G$, $\tau_g = \psi_{h} \lsuprsub{h}{\varphi}{{h^{-1}nh}} = \varphi_n \lsuprsub{n}{\psi}{h}$ (see Diagram \eqref{Equ:varphi psi}). For any morphism $f: ((X, \varphi), \psi) \longrightarrow ((X', \varphi'), \psi')$, define $F(f):= f$. We need to show $(X,\tau)$ and $F(f)$ are in $\gauge{\B}{G}{G}$.

For any $g_1 = n_1h_1, g_2 = n_2h_2 \in G, $ we have

\begin{align*}
\tau_{g_1}\lsuprsub{{g_1}}{\tau}{{g_2}} =& \varphi_{n_1}\lsuprsub{{n_1}}{\psi}{{h_1}} \lsup{{n_1h_1}}{{(\varphi_{n_2}\lsuprsub{{n_2}}{\psi}{{h_2}})}}\\ =& \varphi_{n_1} \lsuprsub{{n_1}}{\psi}{{h_1}} \lsuprsub{{n_1h_1}}{\varphi}{{n_2}} \lsuprsub{{n_1h_1n_2}}{\psi}{{h_2}}\\  \xlongequal{\operatorname{Equ} \, \ref{Equ:varphi psi}}& \varphi_{n_1} \lsuprsub{{n_1}}{\varphi}{{h_1n_2h_1^{-1}}} \lsuprsub{{n_1h_1n_2h_1^{-1}}}{\psi}{{h_1}} \lsuprsub{{n_1h_1n_2}}{\psi}{{h_2}}\\ \xlongequal{\operatorname{Equ} \, \ref{Equ: varphi},\ref{equ:psi}} &\varphi_{n_1h_1n_2h_1^{-1}} \lsuprsub{{n_1h_1n_2h_1^{-1}}}{\psi}{{h_1h_2}}\\ = &\tau_{g_1g_2}.
\end{align*}

This shows $(X,\tau)$ is an object of $\gauge{\B}{G}{G}$. Next, we need to show if $f: ((X, \varphi), \psi) \longrightarrow ((X', \varphi'), \psi')$, then $f: (X,\tau) \longrightarrow (X', \tau')$ is also a morphism in $\gauge{\B}{G}{G}$. This is justified by the following diagram.

\begin{equation}
\xymatrix{
\lsup{{nh}}{X} \ar[rr]^-{\tau_{nh}} \ar[ddd]_-{\lsup{{nh}}{f}} \ar[dr]^-{\lsuprsub{n}{\psi}{h}}   &   &  X \ar[ddd]^-{f} \\
 &  \lsup{n}{X} \ar[ur]^-{\varphi_{n}} \ar[d]^-{\lsup{n}{f}} & \\
 &  \lsup{n}{{X'}} \ar[dr]^-{\varphi'_{n}} &\\
 \lsup{{nh}}{{X'}} \ar[ur]^-{\lsuprsub{n}{{\psi'}}{h}} \ar[rr]_-{\tau'_{nh}} & & X' \\
}
\end{equation}

In the above diagrams, the two triangles both commute by the definition of $\tau, \tau'$. The left and right trapezoids commute since $f$ is both an  $N$-equivariant and an $H$-equivariant morphism. Therefore, the rectangle commutes which shows that $f$ is a $G$-equivariant morphism. 

%And thus clearly it follows that $F$ is a functor.

Next we show that $F$ is a tensor functor and preserves the braiding.

For simplicity, we will also write $((X,\varphi),\psi)$ as $(X, \varphi, \psi)$.
\begin{align*}
F((X,\varphi,\psi) \otimes (X',\varphi',\psi')) &= F(X \otimes X', \varphi \otimes \varphi', \psi \otimes \psi')\\ &= (X \otimes X', \{(\varphi \otimes \varphi')_n \lsuprsub{n}{{(\psi \otimes \psi')}}{h}\}_{nh = g \in G})\\ &= (X \otimes X', \tau_g \otimes \tau'_g)\\ &= F(X,\varphi,\psi) \otimes F(X',\varphi',\psi').
\end{align*}

Note that some of \lq\lq=" signs in the above equations actually represent canonical isomorphisms, such as $\lsup{h}{{(\lsup{g}{X})}} \cong \lsup{{hg}}{X}$. So, there is a canonical isomorphism from $F((X,\varphi,\psi) \otimes (X',\varphi',\psi'))$ to $F(X,\varphi,\psi) \otimes F(X',\varphi',\psi')$, and it is straightforward to check this isomorphism preserves the associativity and thus $F$ is a tensor functor. Recall that the $H$-crossed braiding on $(\cross{\B}{G})^{N}$ is given in Lemma \ref{lem:gauge N} as follows: for $(X,\varphi), (X', \varphi') \in (\cross{\B}{G})^{N}, X \in \B_{hn},$

$$X \otimes X' \overset{c}{\longrightarrow} \lsup{{hn}}{{X'}} \otimes X \overset{\lsuprsub{h}{{\varphi'}}{n} \otimes \operatorname{id}}{\longrightarrow} \lsup{h}{{X'}} \otimes X.$$
Therefore, for $(X,\varphi,\psi) , (X',\varphi',\psi') \in [\gauge{\B}{G}{N}]^{H}$, the braiding is given by the $H$-crossed braiding followed by $\psi'_h \otimes \operatorname{id}$, and is thus equal to the following compositions:
$$X \otimes X' \overset{c}{\longrightarrow} \lsup{{hn}}{{X'}} \otimes X \overset{\lsuprsub{h}{{\varphi'}}{n} \otimes \operatorname{id}}{\longrightarrow} \lsup{h}{{X'}} \otimes X \overset{\psi'_h \otimes \operatorname{id}}{ \longrightarrow} X' \otimes X.$$
The image of this composition under $F$ is $(\psi'_h \lsuprsub{h}{{\varphi'}}{n} \otimes \operatorname{id}) c = (\tau'_{hn} \otimes \operatorname{id})c$, which is exactly the braiding $(X,\tau) \otimes (X',\tau') \longrightarrow (X',\tau') \otimes (X,\tau)$ in $\gauge{\B}{G}{G}$. Therefore, $F$ is a braided tensor functor.

Conversely, given $(X,\tau) \in \gauge{\B}{G}{G},$ define $K(X,\tau):= ((X, \tau_{|_{N}}), \tau_{|_{H}}).$ It is clear that $((X, \tau_{|N}), \tau_{|H})$ satisfy (\ref{Equ: varphi}), (\ref{Equ:varphi psi}) and (\ref{equ:psi}).  Thus  $K(X,\tau)$ is an object of $\gauge{(\gauge{\B}{N}{N})}{H}{H}$. It is routine to check that $K$ is a braided tensor functor and that $FK \simeq \operatorname{Id}, KF \simeq \operatorname{Id}$.

\end{proof}

\end{thm}

%% file: obstruction.tex
\section{Obstructions}\label{section obstruction}

\subsection{$H^3$ obstruction}
Given a group homomorphism $\rho: G\to \Pic{\B}\cong \Aut{\B}$ a necessary condition for the existence of a 
gauging associated to $\rho$, is the existence of a lifting $\u{\rho}: \u{G}\to \u{\Aut{\B}}$ of $\rho$. So in this subsection we will describe some formulas for the computation of the $H^3$-obstruction associated with a group homomorphism $\rho: G\to \operatorname{Aut}_\otimes(\B)$.

Let $\C$ be a  fusion category and $$\widehat{K_0(\C)}=\{f:K_0(\C)\to U(1): f(X\otimes Y)=f(X)f(Y), \  \  \forall X, Y\in  \operatorname{Irr}(\B)\}.$$ Thus $\widehat{K_0(\C)}$ is an abelian group and for every tensor autoequivalence $F\in \operatorname{Aut}_\otimes(\C)$, the abelian group $\operatorname{Aut}_\otimes(\operatorname{Id}_F)$ can be canonically identified with $\widehat{K_0(\C)}$. 

Let  $\rho: G\to \operatorname{Aut}(\C)$ be a group homomorphism. Note that $G$ acts on $\widehat{K_0(\C)}$ since $G$ acts on $K_0(\C)$. Let us fix a representative tensor autoequivalence $F_g:\C\to \C$ for each $g\in G$ and a  tensor natural isomorphism $\theta_{g,h}: F_g\circ F_{h}\to F_{gh}$ for each pair $g,h \in G$, we can assume that $F_e=\operatorname{Id}_\C$ and $\theta_{g,e}=\theta_{e,g}= \operatorname{Id}_{F_g}$ for all $g\in G$. Define $O_3(\rho)\in Z^3(G, \widehat{K_0(\C)})$ by the diagram

\begin{equation}\label{definition 3-cocycle}
\begin{gathered} 
\xymatrixcolsep{5pc} \xymatrix{
F_g\circ F_h\circ F_l  \ar[dd]^{F_g(\theta_{h,l})} \ar[r]^{(\theta_{g,h})_{F_l}}& F_{gh}\circ F_l \ar[d]^{\theta_{gh,l}}\\
 &F_{ghl}\ar[d]^{O_3(\rho)(g,h,l)}\\
F_g\circ F_{hl} \ar[r]^{\theta_{g,hl}} &F_{ghl}.
 }
\end{gathered}
\end{equation}

The proof of the following proposition is straightforward, see \cite[Theorem 5.5]{Ga1}.

\begin{prop}\label{obstruction 3}
Let $\C$ be a fusion category and  $\rho: G\to \operatorname{Aut}_\otimes(\C) $ a group morphism. The cohomology class of the 3-cocycle  $O_3(\rho)$ defined by the diagram \eqref{definition 3-cocycle} only depends on $\rho$. The map $\rho$ lifts to an action $\underline{\rho}:\underline{G}\to \u{\operatorname{Aut}_\otimes(\C)}$ if and only if $0=[O_3(\rho)]\in H^3(G,\widehat{K_0(\C)} )$. In case $[O_3(\rho)]=0$ the set of equivalence classes of liftings of $\rho$ is a torsor over $H^2(G,\widehat{K_0(\C)})$.
\end{prop}
\qed

\begin{rem}
An analogous result holds if $\B$ is a braided fusion category and $\rho: G\to \Aut \B$. In this case there is a third cohomology class $O_3(\rho) \in H^3(G, \widehat{K_0(\B)})$ and equivalence classes of liftings of $\rho$ form a torsor over $H^2(G,\widehat{K_0(\B)})$.
\end{rem}

\subsubsection{Obstruction for pointed braided fusion categories}

Let $\B= \operatorname{Vec}_A^{\omega,c}$ be a braided pointed fusion category. The map 
\begin{align*}
q:A &\to U(1)\\
a &\mapsto c(a,a)
\end{align*}
is a quadratic form and the pair $(A,q)$, called a pre-metric group, is a complete invariant of the equivalence class of $\B$, \cite{JS,DGNO2}. We will denote by $O(A,q)$ the group of all group automorphisms of $A$ that fix $q$.

A braided autoequivalence $(\rho,\psi):\B\to \B$ is defined by a group isomorphism $\rho:A\to A$ and 2-cochain $\psi\in C^2(A,U(1))$ such that 

\begin{align}
\frac{\omega(a,b,c)}{\omega(\rho(a),\rho(b),\rho(c))} &= \frac{\psi(b,c)\psi(a,bc)}{\psi(ab,c)\psi(a,b)} \label{eq 1}\\
\frac{c(a,b)}{c(\rho(a),\rho(b))}&= \frac{\psi(a,b)}{\psi(b,a)} \label{eq 2}
\end{align}for all $a,b,c\in A$. 
Note that for every braided tensor autoequivalence $(\rho,\psi)$, $\rho\in O(A,q)$. Conversely, for every $\rho \in O(A,q)$ there is $\psi\in C^2(A,U(1))$ such that $(\rho,\psi)$ is a braided autoequivalence and the tensor functor  $(\rho,\psi)$ is unique up to tensor equivalence, \cite{JS}. 

Given a group homomorphism $\rho: G\to O(A,q)$,  let us fix for every $g\in G$ a 2-cochain $\psi_g\in C^2(A,U(1))$ such that $(\rho(g),\psi_g)\in \Aut \B$, that is a map $$\psi: G\to C^2(A,U(1)),$$ such that $\psi_g$ satisfies the equations \eqref{eq 1} and \eqref{eq 2} for each $g\in G$.
For every pair $g, h \in G$ fix a tensor natural isomorphism $\theta(g,h): (\rho(g),\psi_g)\circ (\rho(h),\psi_h)\to (\rho(gh) ,\psi_{gh})$, 
that is a map $$\theta: G\times G \to C^1(A,U(1))$$ such that

$$\delta_{G}(\psi)= \delta_A(\theta)^{-1}$$
Now, define 
\begin{equation}\label{obstruction 3 pointed}
O_3(\rho):=\delta_{G}(\theta),
\end{equation} 
then $O_3(\rho)\in Z^3(G,C^1(A,U(1))).$ But since 
\begin{align}
\delta_A(O_3(\rho)) &= \delta_A \Big (\delta_G(\theta) \Big )\notag\\
&= \delta_G \Big (\delta_A(\theta) \Big )\notag\\
&= \delta_G\Big (\delta_G(\psi)^{-1} \Big )=1,\notag
\end{align}
thus $O_3(\rho)\in Z^3(G,\widehat{A}).$ The cohomology class of $O_3(\rho)$ is just the cohomology of Proposition \ref{obstruction 3}.

We summarize the results in the following proposition:

\begin{prop}
Let $\B$ be a pointed braided fusion category with associated pre-metric groups $(A,q)$. Then 

\begin{itemize}
\item $\Aut{\B}= O(A,q)$.
\item A representative 3-cocycle for the $H^3$-obstruction is given by formula \eqref{obstruction 3 pointed}
\end{itemize}
\end{prop}
\qed
\begin{cor}\label{sufficient cond for trivial 3-obs}
Let $c:A\times A\to U(1)$ be a bicharacter and $\B=\operatorname{Vec}_A^{c}$ the associated pointed braided fusion category. Let $$O(A,c)=\{g\in \operatorname{Aut}(A): c(a,b)= c(g(a),g(b))\  \forall a,b, \in A\}.$$ Then every group homomorphism $\rho: G\to O(A,c) \subset O(A,q)$ has trivial $H^3$-obstruction. 
\end{cor}
\begin{proof}
Since $\rho(g)\in (A,c)$, then $(\rho(g),1)\in \Aut{\operatorname{Vec}_A^{c}}$ and $\theta_{g,h}=1$ define a canonical  categorical action $\rho:\u{G}\to \u{\Aut\B}$.
\end{proof}
\begin{cor}
If $A$ is an abelian group of odd order then  for every group homomorphism $\rho: G\to O(A,q)$ the obstruction $O_3(\rho)$ vanishes. 
\end{cor}
\begin{proof}
If $A$ has odd order and $q$ is a quadratic form, then there is a symmetric bicharacter $c_q:A\times A\to U(1)$ such that $q(a)=c_q(a,a)$ and  $O(A,q)=O(A,c_q)$. Thus by Corollary \ref{sufficient cond for trivial 3-obs} the $H^3$-obstruction vanishes. 
\end{proof}

%\begin{rem}
%\begin{itemize}
%\item Let $\B$ be a pointed braided fusion category. Given a group morphism $\rho:G\to \Aut{\B}$, if $\rho(g)$s do not permute the simple objects of $\B$ then $\rho(g)$ is the identity.  Therefore, if it holds for every $g$ in $G$, the map $\rho$ is trival and its $H^3$-obstruction vanishes.
%\end{itemize}
%\end{rem}
\subsection{$H^4$-obstruction}

As we mentioned in Section \ref{gauging}, by Theorem \ref{ENO3712}, faithfully graded $G$-crossed braided fusion categories are in one-one correspondence with  tri-homomorphisms $\uu{\rho}: \uu{G}\rightarrow \uu{\Pic \B}$, or equivalently with maps between their classifying spaces $ \text{BG} \to \text{B}\uuPic{B}$. We 
are interested in the case that $\B$ is modular, so the first truncation of $\uuPic{B}$, denoted as $\uPic{B}$, is monoidal equivalent to $\u{\Aut{\B}}$, \cite[Theorem 5.2]{ENO3}.

\subsection{Obstruction theory for quasi-trivial extensions}
A $G$-graded fusion category $\C=\oplus_{g\in G}\C_g$ is called a quasi-trivial extension of $\C_e$ by $G$, if each homogeneous component $\C_g$ has at least one multiplicatively invertible object.
Let us recall briefly the classification of quasi-trivial extensions, given in \cite{G}. For a fusion category $\C$, we shall denote by $\underline{\underline{\operatorname{Out}_\otimes(\C)}}$ the 3-group of outer autoequivalences:

\begin{itemize}
\item objects are tensor autoequivalence of $\C$,
\item 1-morphisms are pseudo-natural transformations,
\item 2-morphisms are modifications.
\end{itemize}
The main result of \cite{G} is a one-to-one correspondence between equivalence classes of quasi-trivial extensions and equivalence classes of homomorphism of 3-groups $\underline{\underline{\rho}}:\underline{\underline{G}}\to \underline{\underline{\operatorname{Out}_\otimes(\C)}}$, where $\underline{\underline{G}}$ is the discrete 3-category where objects are the elements of $G$. Explicitly a datum for a tri-homomorphism corresponds to 

\begin{itemize}
  \item tensor autoequivalences $(g_*,\psi^{g}): \C\to \C$ for all $g \in G$,
  \item pseudonatural isomorphisms $(\omega(g,h),\chi_{g,h}): g_*\circ h_*\to (gh)_*$ for all $g,h  \in G$,
  \item invertible modifications $\eta_{g,h,l}: \chi_{g,hl}\overline{\circ} (\id_{g_*} \overline {\otimes}\chi_{h,l}) \to \chi_{gh,l}\overline{\circ} (\chi_{g,h}\overline{\otimes}\id_{l_*})$ for all $g,h,l \in G$.
\end{itemize}
such that the diagram in Figure \ref{coherencia} commutes for all $g,h,k,l \in G$ (where tensor symbols among objects and arrows have been omitted).

\begin{figure}
\includegraphics{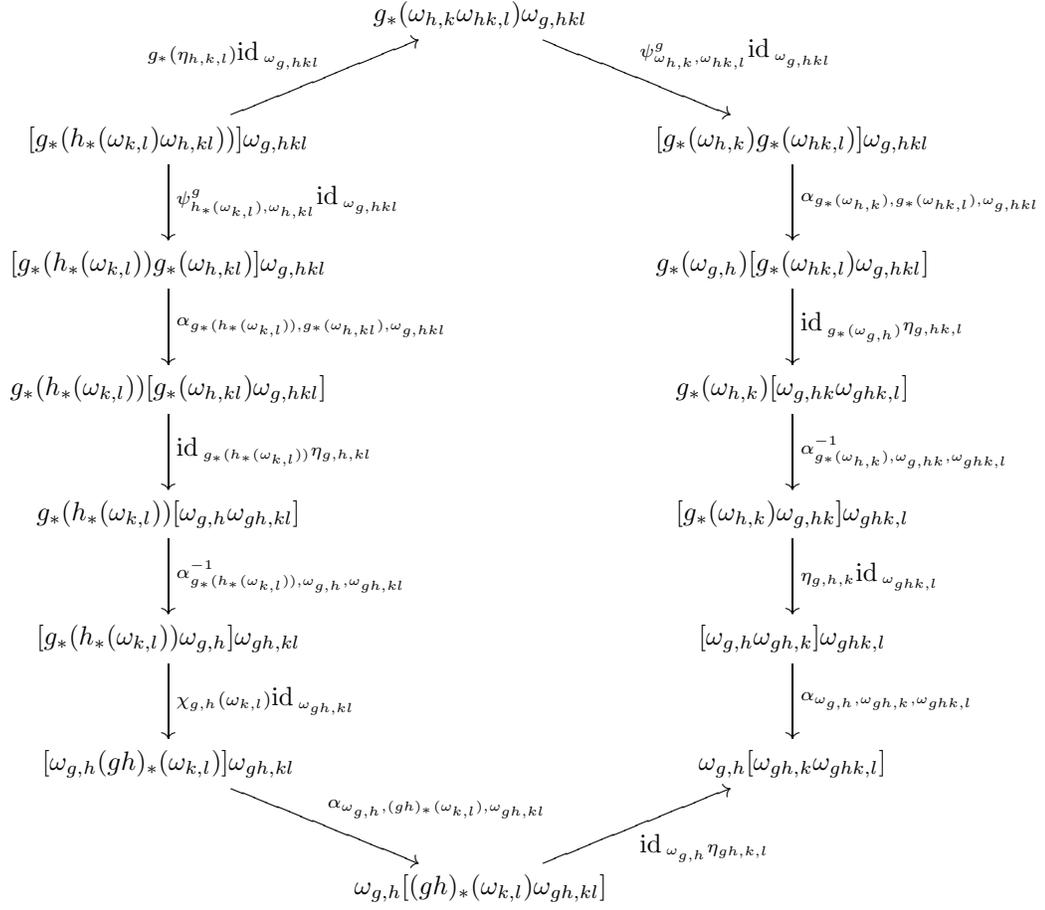}
\caption{coherence for tri-homomorphism}\label{coherencia}
\end{figure}

Let $\underline{\rho}:\underline{G}\to \underline{\operatorname{Out}_\otimes(\C)}$ be a monoidal functor, the  obstruction to the existence of a lifting $\underline{\underline{\rho}}:\underline{\underline{G}}\to \underline{\underline{\operatorname{Out}_\otimes(\C)}}$ is an element in $H^4(G,U(1))$, defined by the next formula:

\begin{align}
O_4(\u\rho)(g_1,g_2,g_3,g_4)=&\chi_{g_1,g_2}(\omega_{g_3,g_4}) \label{formula 2}\\
&\psi^{g_1}_{(g_2)_*(\omega_{g_3,g_4}),\omega_{g_2,g_3g_4}}\notag\\ &(\psi^{g_1}_{\omega_{g_2,g_3},\omega_{g_{2}g_{3},g_{4}}})^{-1} \notag\\
& \alpha_{(g_1)_*((g_2)_*(\omega_{g_3,g_4})),(g_1)_*(\omega_{g_2,g_{3}g_{4}}),\omega_{g_{1},g_{2}g_{3}g_{4}}}\notag\\
&\alpha_{(g_1)_*((g_2)_*(\omega_{g_3,g_4})),\omega_{g_1,g_2},\omega_{g_1g_2,g_3g_4}}^{-1} \notag\\
&\alpha_{\omega_{g_1,g_2},(g_1g_2)_*(\omega_{g_3,g_4}),\omega_{g_1g_2,g_3g_4}}\notag\\
& \alpha_{\omega_{g_1,g_2},\omega_{g_1g_2,g_3},\omega_{g_1g_2g_3,g_4}}^{-1}\notag\\
&\alpha_{(g_1)_*(\omega_{g_2,g_3}),\omega_{g_1,g_2g_3},\omega_{g_1g_2g_3,g_4}}\notag\\
&\alpha_{(g_1)_*(\omega_{g_2,g_3}),(g_1)_*(\omega_{g_2g_3,g_4}), \omega_{g_1,g_2g_3g_4}}^{-1}\notag
\end{align}

The function $O_4(\u\rho): G^{\times 4}\to U(1)$ is a 4-cocycle and its cohomology class only depends on the equivalence class of  $\underline{\rho}: \underline{G}\to \underline{\operatorname{Out}_\otimes(\C)}$.

\subsection{Quasi-trivial extension of a group by a braided fusion category}

\begin{definition}
Following \cite{DGNO2}, we say that a quasi-trivial $G$-extension $\B$ is a braided quasi-trivial extension of $G$ if $\B_e$ is a BFC and for each $g\in G$, the tensor autoequivalence
\begin{align*}
\operatorname{Ad}_{X_g}: \B_e&\to \B_e\\
W &\mapsto (X_g\otimes W)\otimes X_g^*
\end{align*}
is a braided equivalence for all invertible objects in $\B_g$ and all $g\in G$.
\end{definition}

If $\B$ is a braided tensor category every inner tensor autoequivalence is naturally isomorphic to the identity functor, so every monoidal functor $ \underline{G}\to \underline{\operatorname{Out}_\otimes(\B)}$, defines a unique monoidal functor $ \underline{G}\to \underline{\operatorname{Aut}_\otimes(\B)}$. 
A monidal functor $\rho: \underline{G}\to \underline{\operatorname{Out}_\otimes(\B)}$ is a lifting of  $\tau: \underline{G}\to \underline{\operatorname{Aut}_\otimes(\B)}$ if $\tau$ is the functor obtained from $\rho$.

Although $\pi_1(\underline{\operatorname{Aut}_\otimes(\B)})=\pi_1(\underline{\operatorname{Out}_\otimes(\B)})$, they are different categorical groups since $\pi_2(\underline{\operatorname{Aut}_\otimes(\B)})=\operatorname{Aut}(\operatorname{Id}_{\B})$ and $\pi_2(\underline{\operatorname{Out}_\otimes(\B)})=\Inv(\mathcal{Z}(\B))$.

Since $(\B,c)$ is braided, the inclusion 
\begin{align*}
\Inv(\B) &\to \Inv(\mathcal{Z}(\B))\\
V &\mapsto (V, c_{-,V}),
\end{align*}
is a splitting of the exact sequence $$0\to \operatorname{Aut}_\otimes(Id_\B)\to \Inv(\mathcal{Z}(\B)) \to \Inv(\B)\to 0,$$so  $\Inv(\mathcal{Z}(\B))=\operatorname{Aut}_\otimes(\operatorname{Id}_\B)\oplus \Inv(\B).$ Moreover, the action of  $\operatorname{Aut}_\otimes^{br} (\B)$ on $\Inv(\mathcal{Z}(\B))=\operatorname{Aut}_\otimes(\operatorname{Id}_\B)\oplus \Inv(\B)$  is compatible with the direct sum, that is $$\Inv(\mathcal{Z}(\B))=\operatorname{Aut}_\otimes(\operatorname{Id}_\B)\oplus \Inv(\B)$$ as  $\operatorname{Aut}_\otimes^{br} (\B)$-module. 

Summarizing the above discussion, we have:
\begin{prop}
There is a bijective correspondence between equivalence classes of liftings $\underline{G}\to \underline{\operatorname{Out}_\otimes^{br}(\B)}$ of a fix monoidal functor $\u{\rho}:\underline{G}\to \underline{\operatorname{Aut}_\otimes^{br}(\B)}$ and elements in $H^2(G,\Inv(\B))$.
\end{prop}
\qed

\begin{prop}\label{H4formula}
Given a categorical action $\underline{\rho}=(g_*, \psi^{g},\theta_{g,h}):\underline{G}\to \underline{\operatorname{Aut}_\otimes^{br}(\B)}$ and a 2-cocycle $\omega \in Z^2(G,\Inv(\B))$, a representative  4-cocycle for the $H^4$-obstruction is given by 
\begin{align}
O_4(\u\rho,\omega)(g_1,g_2,g_3,g_4)=&\theta_{g_1,g_2}(\omega_{g_3,g_4}) \label{formula 3}\\
&c_{(g_1g_2)_*(\omega_{g_3,g_4}),\omega_{g_1,g_2}}\notag\\
&\psi^{g_1}_{(g_2)_*(\omega_{g_3,g_4}),\omega_{g_2,g_3g_4}}\notag\\ &(\psi^{g_1}_{\omega_{g_2,g_3},\omega_{g_{2}g_{3},g_{4}}})^{-1} \notag\\
& \alpha_{(g_1)_*((g_2)_*(\omega_{g_3,g_4})),(g_1)_*(\omega_{g_2,g_{3}g_{4}}),\omega_{g_{1},g_{2}g_{3}g_{4}}}\notag\\
&\alpha_{(g_1)_*((g_2)_*(\omega_{g_3,g_4})),\omega_{g_1,g_2},\omega_{g_1g_2,g_3g_4}}^{-1} \notag\\
&\alpha_{\omega_{g_1,g_2},(g_1g_2)_*(\omega_{g_3,g_4}),\omega_{g_1g_2,g_3g_4}}\notag\\
& \alpha_{\omega_{g_1,g_2},\omega_{g_1g_2,g_3},\omega_{g_1g_2g_3,g_4}}^{-1}\notag\\
&\alpha_{(g_1)_*(\omega_{g_2,g_3}),\omega_{g_1,g_2g_3},\omega_{g_1g_2g_3,g_4}}\notag\\
&\alpha_{(g_1)_*(\omega_{g_2,g_3}),(g_1)_*(\omega_{g_2g_3,g_4}), \omega_{g_1,g_2g_3g_4}}^{-1}\notag
\end{align}
\end{prop}
\begin{proof}
The pseudo-natural transformation associated to $\omega \in Z^2(G,\Inv(\B))$ is $(\omega(g_1,g_2),\chi_{g_1,g_2}): (g_1)_*\circ (g_2)_*\to (g_1g_2)_*$ defined by $$\chi_{g_1,g_2}(V):= c_{(g_1g_2)_*(V),\omega(g_1,g_2)}\circ(\theta_{g_1,g_2}(V)\otimes \id_{\omega(g_1,g_2)}),$$ for all $V\in \C, g_1,g_2\in G$.

Hence replacing  $(\omega(g_1,g_2),\chi_{g_1,g_2})$ in formula \eqref{formula 2}, we get the new formula of the 4-cocycle.
\end{proof}

\subsection{$H^4$ obstruction to $G$-crossed braided fusion categories}

Let $(\B,c)$ be a BFC. Suppose a categorical action $(g_*,\psi^g,\theta_{g,h})_{(g,h\in G)}: \u{G}\to \u{\Aut \B}$ 
admits a gauging $\uu{\rho}:\uu{G}\to \uu{\Pic \B}$. 
Then the equivalence classes of  homomorphism of 2-groups $ \u{G}\to \u{\Pic \B}$ with associated topological symmetry $(g_*,\psi^g,\theta_{g,h})_{(g,h\in G)}$ is a 
torsor over $H^2_\rho(G, \Inv(\B))$. 
Given an element $\omega \in Z^2_\rho(G, \Inv(\B)))$, we shall denote by $(\omega\rhd \u{\rho}):\u{G}\to \u{\Pic \B}$, the associated homomorphism of 2-groups.

\begin{prop}
The homomorphism  $(\omega\rhd \u{\rho}):\u{G}\to \u{\Pic \B}$ can be gauged if and only if the 4-cocycle
$O_4(\u\rho,\omega)$, defined in equation \eqref{formula 3}, is cohomologically trivial.
\end{prop}
\begin{proof}
The obstruction $O_4(\u\rho,\omega)$ is a concrete formula for the Pontryagin-Whitehead quadratic function defined in \cite[Section 8.7]{ENO3}, so the proposition follows from \cite[Proposition 8.15]{ENO3}.
\end{proof}

\begin{cor}\label{formula 4 quasi-trivial}
Let  $t:\u{G}\to \u{\Aut \B}$ be the trivial homomorphism and $\omega\in Z^2(G,\Inv(\B))$. The homomorphism $(\omega\rhd t):\u{G}\to \u{\Pic \B}$ can be gauged if and only if the cohomology class of the 4-cocycle 

\begin{align*}
O_4(g_1,g_2,g_3,g_4)=&c(\omega_{g_3,g_4},\omega_{g_1,g_2})\\
&\alpha_{\omega_{g_3,g_4} ,\omega_{g_2,g_3g_4},\omega_{g_1,g_2g_3g_4}}\\
&\alpha_{\omega_{g_3,g_4},\omega_{g_1,g_2},\omega_{g_1g_2,g_3g_4}}^{-1}\\
&\alpha_{\omega_{g_1,g_2},\omega_{g_3,g_4},\omega_{g_1g_2,g_3g_4}}\\
&\alpha_{\omega_{g_1,g_2},\omega_{g_1,g_2g_3},\omega_{g_1g_2g_3,g_4}}^{-1}\\
&\alpha_{\omega_{g_2,g_3},\omega_{g_1,g_2g_3},\omega_{g_1g_2g_3,g_4}}\\
&\alpha_{\omega_{g_2,g_2},\omega_{g_2g_3,g_4},\omega_{g_1,g_2g_3g_4}}^{-1},
\end{align*}
vanishes. 
\end{cor}

The same formula in this case was derived in \cite{BBCW}.

If the topological symmetry $(\omega\rhd t):\u{G}\to \u{\Aut \B}$ can be gauged the associated $G$-crossed braided fusion categories are quasi-trivial extensions. Conversely, every quasi-trivial $G$-crossed braided fusion category is the gauging of a topological symmetry $(\omega\rhd t):\u{G}\to \u{\Aut \B}$.

Despite the simplicity of the topological symmetry $(\omega\rhd t):\u{G}\to \u{\Aut \B}$, the following proposition said that an interesting family of UMCs can be obtained as gaugings.

\begin{prop}\label{GTPMC}
Every group-theoretical modular tensor category is the gauging of a topological symmetry $(\omega\rhd t): \u{G}\to \uAut{\B}$, where $\B$ is a pointed modular tensor category.
\end{prop}
\begin{proof}
Recall that an equivariantization of a $G$-crossed braided fusion category $\B$ is modular if and only if the $G$-grading is faithful and $\B_e$ is modular.

By \cite[Theorem 5.3]{Na} every braided group-theoretical fusion category $\B$ can be obtained as a gauging of a pointed $G$-crossed braided fusion category $\C$. The pair $(G,\Inv(\C))$ is an ordinary  crossed module, where the $G$-action on $X$ is induced by the $G$-action on $\C$ and the morphism  $\partial: X\to G$ is defined by the $G$-grading. 

Since $\B$ is modular, $\partial$ is surjective, so $X$ is a central extension $G$ by $A=\Inv(\C_e)$  and $\C_e$ is a pointed modular category. If $\omega \in Z^2(G,A)$ is a 2-cocycle corresponding to the central extension $X$, then $\B$ is a gauging of the topological symmetry $(*,\omega):\u{G}\to \u{\Aut{\C_e}}$.
\end{proof}

\begin{rem}
\begin{itemize}
\item Every integral modular tensor category of Frobenius-Perron dimension $p^n$, with $p$ a primer number, is group-theoretical, \cite[Theorem 1.5]{DGNO1}, \cite[Theorem 8.28]{ENO}. Every fusion category of dimension $p^n$ with $p$ odd is automatically integral \cite{GN}. 
\item Using Proposition \ref{GTPMC} and Corollary \ref{formula 4 quasi-trivial} we can reduce the classification of group-theoretical modular categories to a pure problem in group cohomology.
\end{itemize}
\end{rem}

%% file: examples.tex
\section{Examples}\label{examples}

An extensive list of examples in the spin-network formalism is given in \cite{BBCW}.  Here we focus on two examples: the $\mathbb{Z}_2$-symmetry of the Deligne product of the Fibonacci category with itself, and the first non-abelian $S_3$ symmetry of the $3$-fermion theory $SO(8)_1$.  It would be interesting to compare our computation with related work in the future \cite{FS, FFRS}.

\subsection{${Fib}^{\boxtimes 2}$ with $\Z_2$ symmetry}
The modular category $\Fib \boxtimes \Fib$ has a $\Z_2$ symmetry which swaps the two $\Fib$ factors. Denote the simple objects (anyons) in $\Fib \boxtimes \Fib $ by $\{\unit = (\unit,\unit), (\unit,\tau), (\tau, \unit),(\tau,\tau)\}$. In the $\Z_2$-crossed braided extension $\cross{(\Fib \boxtimes \Fib)}{\Z_2}$, the sector labelled by the non-trivial element of $\Z_2$ contains two defects (simple objects), which are denoted by $X_{\unit}, X_{\tau}$. Number all the anyons in both sectors in the order $\{\unit = (\unit,\unit), (\unit,\tau), (\tau, \unit),(\tau,\tau),X_{\unit}, X_{\tau}\}$ by $\{1,2,3,4,5,6\}$. Below we give part of the data associated to $\cross{(\Fib \boxtimes \Fib)}{\Z_2}$, and the rest of the data can be found in Appendix %\ref{sec:dataFib}.
A.

The quantum dimensions are:
$$\left\{1,\frac{1}{2} \left(1+\sqrt{5}\right),\frac{1}{2} \left(1+\sqrt{5}\right),\frac{1}{2} \left(3+\sqrt{5}\right),\sqrt{\frac{1}{2}
   \left(5+\sqrt{5}\right)},\sqrt{5+2 \sqrt{5}}\right\}$$
Thus the total quantum dimension is $\mathcal{D} = \frac{\sqrt{10}+5 \sqrt{2}}{2}$.

The Frobenius-Shur indicators are $\{1,1,1,1,-1,-1\}$. So the two defects have non-trivial Frobenius-Shur indicators.

For the group action, $g$ swaps $2$ with $3$, and fixes all other simple objects.

From the group actions, we deduce that the fusion rules are symmetric, namely $a \otimes b = b \otimes a.$  But the category is not braided. We omit the fusion rules of the subcategory $\Fib \boxtimes \Fib$ and those of the trivial object since they are rather simple. Note that some of the fusion rules have multiplicity more than $1$.

\begin{itemize}

%\item $2\: \otimes \: 2\: = \: 1\: + \: 2$

%\item $2\: \otimes \: 3\: = \: 4$

%\item $2\: \otimes \: 4\: = \: 3\: + \: 4$

\item $2\: \otimes \: 5\: = \: 6$

\item $2\: \otimes \: 6\: = \: 5\: + \: 6$

%\item $3\: \otimes \: 3\: = \: 1\: + \: 3$

%\item $3\: \otimes \: 4\: = \: 2\: + \: 4$

\item $3\: \otimes \: 5\: = \: 6$

\item $3\: \otimes \: 6\: = \: 5\: + \: 6$

%\item $4\: \otimes \: 4\: = \: 1\: + \: 2\: + \: 3\: + \: 4$

\item $4\: \otimes \: 5\: = \: 5\: + \: 6$

\item $4\: \otimes \: 6\: = \: 5\: + \: {\color{red}{6}}\: + \: {\color{red}{6}}$

\item $5\: \otimes \: 5\: = \: 1\: + \: 4$

\item $5\: \otimes \: 6\: = \: 2\: + \: 3\: + \: 4$

\item $6\: \otimes \: 6\: = \: 1\: + \: 2\: + \: 3\: + \: {\color{red}{4}}\: + \: {\color{red}{4}}$
\end{itemize}

Appendix %\ref{sec:dataFib}
A 
contains a complete list of the remaining data including $F$-matrices, $R$-matrices, etc.

Gauging the $\Z_2$ symmetry of $\Fib \boxtimes \Fib$ results in the gauged theory $\SU(2)_8$, i.e., $\gauge{(\Fib \boxtimes \Fib)}{\Z_2}{\Z_2}=\SU(2)_8$. The data associated to $\SU(2)_8$ can be found in a number of reference, e.g \cite{PB}. One can verify $\SU(2)_8$  is indeed the correct outcome for gauging by computing the inverse process, which is called taking the core \cite{DGNO2}. Actually, the data for $\cross{(\Fib \boxtimes \Fib)}{\Z_2}$, $F$-matrices, $R$-matrices, etc, is obtained from computing the de-equivariantization of $\SU(2)_8$.

\subsection{$\SO(8)_1$ with the non-abelian $S_3$ symmetry}
The $\SO(8)_1$ theory, also called the $3$-fermion theory, has three mutually fermionic anyons, which are denoted by $\{\psi_1,\psi_2,\psi_3\}$. The fusion rules of the three fermions and the vacuum $\unit$ form the group $\Z_2 \times \Z_2$. Any permutation of the three fermions leaves the theory invariant, thus $\SO(8)_1$ has a symmetry group $\Sym_3$, which is a non-abelian symmetry. Since $\Sym_3 = \Z_3 \rtimes \Z_2$, by Theorem \ref{thm:seq gauge}, in order to gauge the whole symmetry group $\Sym_3$, we can first gauge $\Z_3$, and then gauge $\Z_2$. By \cite{BBCW}, gauging $\Z_3$ results in the theory $\SU(3)_3$, whose data can be found in \cite{AS}. 

The theory $\SU(3)_3$ has $10$ anyon types, which are denoted by $\\ \{\unit, a, a', X, Y, X', aX, aX', a'X, a'X'\}$. We arrange the anyons in the order as shown in Figure \ref{fig:su3}, then the $\Z_2$ symmetry is simply a reflection along the height of the vertical edge of the triangle. The $\Z_2$ extension $\cross{(\SU(3)_3)}{\Z_2}$ of $\SU(3)_3$ contains one defect sector, as well as the trivial sector $\SU(3)_3$. The defector sector contains two defects $\{X^{+}, X^{-}\}$. For the fusion rules of $\cross{(\SU(3)_3)}{\Z_2}$ involving the defects, see \cite{BBCW}. 

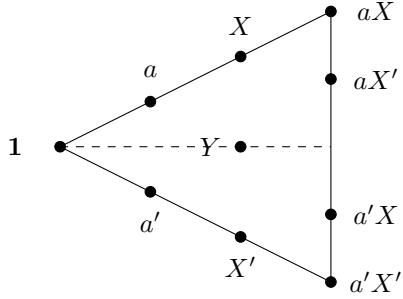
\begin{figure}
\begin{tikzpicture}[scale = 2]
\begin{scope}
\draw (0,0) -- (1.8,0.9) -- (1.8,-0.9) -- (0,0);
\draw[dashed] (0,0) -- (1.8,0);
\filldraw (0,0) circle(1pt);
\filldraw (0.6,0.3) circle(1pt);
\filldraw (1.2,0) circle(1pt);
\filldraw (1.2,0.6) circle(1pt);
\filldraw (1.8,0.45) circle(1pt);
\filldraw (1.8,0.9) circle(1pt);
\filldraw (0.6,-0.3) circle(1pt);
\filldraw (1.2,-0.6) circle(1pt);
\filldraw (1.8,-0.45) circle(1pt);
\filldraw (1.8,-0.9) circle(1pt);

\draw (-0.3,0) node{$\unit$};
\draw (0.6,0.5) node{$a$};
\draw (0.6,-0.5) node{$a'$};
\draw (1.2,0.8) node{$X$};
\draw (1,0) node{$Y$};
\draw (1.2,-0.8) node{$X'$};
\draw (2.1,0.9) node{$aX$};
\draw (2.1,0.45) node{$aX'$};
\draw (2.1,-0.9) node{$a'X'$};
\draw (2.1,-0.45) node{$a'X$};
\end{scope}
\end{tikzpicture}
\caption{Symmetry for $\SU(3)_3$} \label{fig:su3}
\end{figure}

By \cite{BN}, we can compute the fusion rules of $\gauge{(\SU(3)_3)}{\Z_2}{\Z_2}$ from those of $\cross{(\SU(3)_3)}{\Z_2}$ and some cohomology data. %There are two theories resulting from gauging, both of which are new modular tensor categories with $12$ anyon types. The two theories have the same $T$ matrix but different fusion rules and $S$-matrix. The $S, T$ matrices are presented below. 

We denote the anyon types by   $$\{\unit,(\extunit,-1),a,(Y,1),(Y,-1),X,aX,aX',(X^{+},1),(X^{+},-1),(X^{-},1),(X^{-},-1)\}.$$
Their quantum dimensions are respectively $\{1,1,2,3,3,4,4,4,3\sqrt{2},3\sqrt{2},3\sqrt{2},3\sqrt{2}\}$.

In Appendix %\ref{sec:dataSO8},
B of the earlier version, two versions of $\gauge{(\SU(3)_3)}{\Z_2}{\Z_2}$ are presented. But it turns out that these two versions are equivalent by a swap $(Y,1) \leftrightarrow (Y,-1)$, so they are actually the same theory. One can also find the fusion rules there.

The $T$-matrix (the twist) of the anyons is given by $$diag = \left(1,1,1,-1,-1,e^{\frac{2\pi i}{9}},e^{\frac{8\pi i}{9}},e^{-\frac{4\pi i}{9}},e^{\frac{\pi i}{8}},e^{-\frac{7\pi i}{8}},e^{\frac{7\pi i}{8}},e^{-\frac{\pi i}{8}}\right)$$

The $S$-matrix is given by: 

$
\left(
\begin{array}{cccccccccccc}
 1 & 1 & 2 & 3 & 3 & 4 & 4 & 4 & 3 \sqrt{2} & 3 \sqrt{2} & 3 \sqrt{2} & 3 \sqrt{2} \\
 1 & 1 & 2 & 3 & 3 & 4 & 4 & 4 & -3 \sqrt{2} & -3 \sqrt{2} & -3 \sqrt{2} & -3 \sqrt{2} \\
 2 & 2 & 4 & 6 & 6 & -4 & -4 & -4 & 0 & 0 & 0 & 0 \\
 3 & 3 & 6 & -3 & -3 & 0 & 0 & 0 & -3 \sqrt{2} & -3 \sqrt{2} & 3 \sqrt{2} & 3 \sqrt{2} \\
 3 & 3 & 6 & -3 & -3 & 0 & 0 & 0 & 3 \sqrt{2} & 3 \sqrt{2} & -3 \sqrt{2} & -3 \sqrt{2} \\
 4 & 4 & -4 & 0 & 0 & b & c & a & 0 & 0 & 0 & 0 \\
 4 & 4 & -4 & 0 & 0 & c & a & b & 0 & 0 & 0 & 0 \\
 4 & 4 & -4 & 0 & 0 & a & b & c & 0 & 0 & 0 & 0 \\
 3 \sqrt{2} & -3 \sqrt{2} & 0 & -3 \sqrt{2} & 3 \sqrt{2} & 0 & 0 & 0 & 0 & 0 & 6 & -6 \\
 3 \sqrt{2} & -3 \sqrt{2} & 0 & -3 \sqrt{2} & 3 \sqrt{2} & 0 & 0 & 0 & 0 & 0 & -6 & 6 \\
 3 \sqrt{2} & -3 \sqrt{2} & 0 & 3 \sqrt{2} & -3 \sqrt{2} & 0 & 0 & 0 & 6 & -6 & 0 & 0 \\
 3 \sqrt{2} & -3 \sqrt{2} & 0 & 3 \sqrt{2} & -3 \sqrt{2} & 0 & 0 & 0 & -6 & 6 & 0 & 0 \\
\end{array}
\right)
$

%$S$-matrix for Theory 2:

%$
%\left(
%\begin{array}{cccccccccccc}
 %1 & 1 & 2 & 3 & 3 & 4 & 4 & 4 & 3 \sqrt{2} & 3 \sqrt{2} & 3 \sqrt{2} & 3 \sqrt{2} \\
 %1 & 1 & 2 & 3 & 3 & 4 & 4 & 4 & -3 \sqrt{2} & -3 \sqrt{2} & -3 \sqrt{2} & -3 \sqrt{2} \\
 %2 & 2 & 4 & 6 & 6 & -4 & -4 & -4 & 0 & 0 & 0 & 0 \\
 %3 & 3 & 6 & -3 & -3 & 0 & 0 & 0 & 3 \sqrt{2} & 3 \sqrt{2} & -3 \sqrt{2} & -3 \sqrt{2} \\
 %3 & 3 & 6 & -3 & -3 & 0 & 0 & 0 & -3 \sqrt{2} & -3 \sqrt{2} & 3 \sqrt{2} & 3 \sqrt{2} \\
 %4 & 4 & -4 & 0 & 0 & b & c & a & 0 & 0 & 0 & 0 \\
 %4 & 4 & -4 & 0 & 0 & c & a & b & 0 & 0 & 0 & 0 \\
 %4 & 4 & -4 & 0 & 0 & a & b & c & 0 & 0 & 0 & 0 \\
 %3 \sqrt{2} & -3 \sqrt{2} & 0 & 3 \sqrt{2} & -3 \sqrt{2} & 0 & 0 & 0 & 0 & 0 & 6 & -6 \\
 %3 \sqrt{2} & -3 \sqrt{2} & 0 & 3 \sqrt{2} & -3 \sqrt{2} & 0 & 0 & 0 & 0 & 0 & -6 & 6 \\
 %3 \sqrt{2} & -3 \sqrt{2} & 0 & -3 \sqrt{2} & 3 \sqrt{2} & 0 & 0 & 0 & 6 & -6 & 0 & 0 \\
 %3 \sqrt{2} & -3 \sqrt{2} & 0 & -3 \sqrt{2} & 3 \sqrt{2} & 0 & 0 & 0 & -6 & 6 & 0 & 0 \\
%\end{array}
%\right)
%$

where $a = -8 \cos(\frac{2\pi}{9}), b = 4 \sqrt{3} \sin \left(\frac{\pi }{9}\right)-4 \cos \left(\frac{\pi }{9}\right), c = 8 \cos \left(\frac{\pi }{9}\right)$ are the three roots of $-64-48x+x^3$.
%a = -4 \sqrt{2 \left(1+\cos \left(\frac{\pi }{9}\right)-\cos \left(\frac{2 \pi }{9}\right)\right)}

%% file: appendix.tex
\appendix
\section{Data for $\cross{(\Fib \boxtimes \Fib)}{\Z_2}$}
\label{sec:dataFib}

\subsection{$F$-Matrices}
The $F$-matrices are defined by the following figure, where $(n,k,l)$ is the row index and $(m,i,j)$ is the column index. The indexes are listed in the dictionary order. So now the $6j$-symbols are really $10j$-symbols. We omit those $F$-matrices where $a, b$ or $c$ is $1$, in which case the $F$-matrices are the identity matrices. $($These identity matrices could have dimension more than $1$ since now there are multiplicities in the fusion rules.$)$

\setlength{\unitlength}{0.030in}
\begin{picture}(50,40)(0,0)
 \put(20,10){\line(0,-1){10}}
 \put(20,10){\line(1,1){20}}
 \put(20,10){\line(-1,1){20}}
\put(10,20){\line(1,1){10}}

 \put(2,30){$a$}
 \put(22,30){$b$}
 \put(42,30){$c$}
 \put(16,16){$m$}
 \put(22,2){$d$}
 \put(5,17){$i$}
 \put(15,7){$j$}

 \put(35,15){=}
 \put(40,15){$\sum\limits_{n} F_{d;n,k,l;m,i,j}^{abc}$}

 \put(80,10){\line(0,-1){10}}
 \put(80,10){\line(1,1){20}}
 \put(80,10){\line(-1,1){20}}
\put(90,20){\line(-1,1){10}}

 \put(62,30){$a$}
 \put(82,30){$b$}
 \put(102,30){$c$}
 \put(86,15){$n$}
 \put(82,2){$d$}
 \put(92,18){$k$}
\put(82,8){$l$}
\end{picture}

Let $a = \frac{1}{20} \left(5 i \left(-1+\sqrt{5}\right)+\sqrt{10 \left(5+\sqrt{5}\right)}\right)$.

\begin{itemize}
\item $\left(
\begin{array}{c}
 1 \\
\end{array}
\right)\: \: \textrm{for} $

$F^{222}_{1}, F^{223}_{3}, F^{223}_{4}, F^{224}_{3}, F^{225}_{6}, F^{226}_{5}, F^{233}_{2}, F^{233}_{4}, F^{234}_{1}, F^{234}_{4}, F^{235}_{5}, F^{242}_{3}  $

$F^{243}_{1}, F^{243}_{2}, F^{243}_{3}, F^{243}_{4}, F^{255}_{2}, F^{255}_{3}, F^{255}_{4}, F^{265}_{2}, F^{265}_{3}, F^{266}_{1}, F^{324}_{3}, F^{324}_{4}  $

$F^{325}_{5}, F^{333}_{1}, F^{342}_{1}, F^{342}_{4}, F^{343}_{2}, F^{344}_{1}, F^{355}_{2}, F^{355}_{4}, F^{356}_{1}, F^{356}_{3}, F^{365}_{1}, F^{365}_{2}  $

$F^{423}_{1}, F^{423}_{4}, F^{425}_{5}, F^{432}_{2}, F^{432}_{4}, F^{433}_{2}, F^{435}_{5}, F^{442}_{1}, F^{444}_{1}, F^{455}_{2}, F^{455}_{3}, F^{456}_{1}  $

$F^{465}_{1}  $

\item $\left(
\begin{array}{c}
 -1 \\
\end{array}
\right)\: \: \textrm{for} $

$F^{225}_{5}, F^{232}_{3}, F^{232}_{4}, F^{234}_{2}, F^{234}_{3}, F^{235}_{6}, F^{236}_{5}, F^{244}_{1}, F^{244}_{3}, F^{245}_{5}, F^{254}_{5}, F^{256}_{1}  $

$F^{256}_{2}, F^{256}_{3}, F^{265}_{1}, F^{322}_{3}, F^{322}_{4}, F^{323}_{2}, F^{323}_{4}, F^{324}_{1}, F^{324}_{2}, F^{325}_{6}, F^{326}_{5}, F^{332}_{2}  $

$F^{332}_{4}, F^{334}_{2}, F^{335}_{5}, F^{335}_{6}, F^{336}_{5}, F^{342}_{2}, F^{342}_{3}, F^{344}_{2}, F^{345}_{5}, F^{354}_{5}, F^{355}_{3}, F^{356}_{2}  $

$F^{365}_{3}, F^{366}_{1}, F^{422}_{3}, F^{423}_{2}, F^{423}_{3}, F^{424}_{1}, F^{424}_{3}, F^{432}_{1}, F^{432}_{3}, F^{434}_{1}, F^{434}_{2}, F^{442}_{3}  $

$F^{443}_{1}, F^{443}_{2}, F^{452}_{5}, F^{453}_{5}, F^{455}_{1}, F^{522}_{5}, F^{533}_{5}  $

\item $\left(
\begin{array}{c}
 \frac{1}{1+\sqrt{5}}+\frac{1}{2} i \sqrt{\frac{1}{2} \left(5+\sqrt{5}\right)} \\
\end{array}
\right)\: \: \textrm{for} $

$F^{252}_{5}, F^{353}_{5}  $

\item $\left(
\begin{array}{c}
 \frac{1}{4} \left(\sqrt{10-2 \sqrt{5}}-i \left(1+\sqrt{5}\right)\right) \\
\end{array}
\right)\: \: \textrm{for} $

$F^{252}_{6}, F^{353}_{6}  $

\item $\left(
\begin{array}{c}
 \frac{1}{1+\sqrt{5}}-\frac{1}{2} i \sqrt{\frac{1}{2} \left(5+\sqrt{5}\right)} \\
\end{array}
\right)\: \: \textrm{for} $

$F^{253}_{5}, F^{352}_{5}  $

\item $\left(
\begin{array}{c}
 \frac{1}{4} \left(\sqrt{10-2 \sqrt{5}}+i \left(1+\sqrt{5}\right)\right) \\
\end{array}
\right)\: \: \textrm{for} $

$F^{253}_{6}, F^{352}_{6}  $

\item $\left(
\begin{array}{c}
 \frac{1}{4} i \left(1+\sqrt{5}+i \sqrt{10-2 \sqrt{5}}\right) \\
\end{array}
\right)\: \: \textrm{for} $

$F^{262}_{5}, F^{363}_{5}  $

\item $\left(
\begin{array}{c}
 -\frac{1}{4} i \left(1+\sqrt{5}-i \sqrt{10-2 \sqrt{5}}\right) \\
\end{array}
\right)\: \: \textrm{for} $

$F^{263}_{5}, F^{362}_{5}  $

\item $\left(
\begin{array}{c}
 -i \\
\end{array}
\right)\: \: \textrm{for} $

$F^{522}_{6}, F^{523}_{6}, F^{525}_{4}, F^{526}_{1}, F^{532}_{6}, F^{535}_{4}, F^{536}_{1}, F^{542}_{5}, F^{543}_{5}, F^{545}_{1}, F^{552}_{3}, F^{552}_{4}  $

$F^{553}_{2}, F^{553}_{3}, F^{553}_{4}, F^{554}_{1}, F^{554}_{2}, F^{554}_{3}, F^{562}_{3}, F^{563}_{1}, F^{563}_{2}, F^{563}_{3}, F^{565}_{5}, F^{622}_{5}  $

$F^{623}_{5}, F^{626}_{1}, F^{632}_{5}, F^{636}_{1}, F^{645}_{1}, F^{652}_{1}, F^{652}_{2}, F^{654}_{1}, F^{663}_{1}  $

\item $\left(
\begin{array}{c}
 i \\
\end{array}
\right)\: \: \textrm{for} $

$F^{523}_{5}, F^{524}_{5}, F^{532}_{5}, F^{533}_{6}, F^{534}_{5}, F^{545}_{2}, F^{545}_{3}, F^{546}_{1}, F^{552}_{2}, F^{555}_{6}, F^{556}_{5}, F^{562}_{1}  $

$F^{562}_{2}, F^{564}_{1}, F^{625}_{1}, F^{633}_{5}, F^{635}_{1}, F^{652}_{3}, F^{653}_{1}, F^{653}_{2}, F^{653}_{3}, F^{655}_{5}, F^{662}_{1}  $

\item $\left(
\begin{array}{c}
 \frac{1}{4} \left(i \left(-1+\sqrt{5}\right)+\sqrt{2 \left(5+\sqrt{5}\right)}\right) \\
\end{array}
\right)\: \: \textrm{for} $

$F^{525}_{2}, F^{535}_{3}  $

\item $\left(
\begin{array}{c}
 \frac{1}{4} i \left(-1+\sqrt{5}+i \sqrt{2 \left(5+\sqrt{5}\right)}\right) \\
\end{array}
\right)\: \: \textrm{for} $

$F^{525}_{3}, F^{535}_{2}  $

\item $\left(
\begin{array}{c}
 \frac{1}{4} \left(-1-\sqrt{5}+i \sqrt{10-2 \sqrt{5}}\right) \\
\end{array}
\right)\: \: \textrm{for} $

$F^{526}_{2}, F^{536}_{3}, F^{625}_{2}, F^{635}_{3}  $

\item $\left(
\begin{array}{c}
 \frac{1}{4} \left(1+\sqrt{5}+i \sqrt{10-2 \sqrt{5}}\right) \\
\end{array}
\right)\: \: \textrm{for} $

$F^{526}_{3}, F^{536}_{2}, F^{625}_{3}, F^{635}_{2}  $

\item $\left(
\begin{array}{cc}
 \frac{1}{2} \left(-1+\sqrt{5}\right) & \sqrt{\frac{1}{2} \left(-1+\sqrt{5}\right)} \\
 \sqrt{\frac{1}{2} \left(-1+\sqrt{5}\right)} & \frac{1}{2} \left(1-\sqrt{5}\right) \\
\end{array}
\right)\: \: \textrm{for} $

$F^{222}_{2}, F^{224}_{4}, F^{266}_{2}, F^{333}_{3}, F^{345}_{6}, F^{433}_{4}  $

\item $\left(
\begin{array}{cc}
 -\sqrt{\frac{2}{3+\sqrt{5}}} & -\sqrt{\frac{1}{2} \left(-1+\sqrt{5}\right)} \\
 \sqrt{\frac{1}{2} \left(-1+\sqrt{5}\right)} & -\sqrt{\frac{1}{2} \left(3-\sqrt{5}\right)} \\
\end{array}
\right)\: \: \textrm{for} $

$F^{226}_{6}, F^{244}_{4}, F^{343}_{4}, F^{442}_{4}, F^{444}_{2}, F^{445}_{5}  $

\item $\left(
\begin{array}{cc}
 -\frac{1}{2} \sqrt{1+\sqrt{5}} & -\frac{-1+\sqrt{5}}{2 \sqrt{2}} \\
 \frac{\sqrt{2}}{1+\sqrt{5}} & -\frac{1}{2} \sqrt{1+\sqrt{5}} \\
\end{array}
\right)\: \: \textrm{for} $

$F^{236}_{6}, F^{346}_{5}  $

\item $\left(
\begin{array}{cc}
 \frac{1}{2} \left(1-\sqrt{5}\right) & \sqrt{\frac{1}{2} \left(-1+\sqrt{5}\right)} \\
 -\sqrt{\frac{1}{2} \left(-1+\sqrt{5}\right)} & \frac{1}{2} \left(1-\sqrt{5}\right) \\
\end{array}
\right)\: \: \textrm{for} $

$F^{242}_{4}, F^{245}_{6}, F^{265}_{4}, F^{344}_{4}, F^{443}_{4}, F^{444}_{3}  $

\item $\left(
\begin{array}{cc}
 \frac{1}{2} \left(-1+\sqrt{5}\right) & -\sqrt{\frac{1}{2} \left(-1+\sqrt{5}\right)} \\
 -\sqrt{\frac{1}{2} \left(-1+\sqrt{5}\right)} & -\sqrt{\frac{2}{3+\sqrt{5}}} \\
\end{array}
\right)\: \: \textrm{for} $

$F^{244}_{2}, F^{365}_{4}, F^{443}_{3}  $

\item $\left(
\begin{array}{cc}
 \frac{\sqrt{1+\sqrt{5}}}{2} & -\frac{1}{2} \sqrt{3-\sqrt{5}} \\
 \frac{1}{\sqrt{3+\sqrt{5}}} & \frac{\sqrt{1+\sqrt{5}}}{2} \\
\end{array}
\right)\: \: \textrm{for} $

$F^{246}_{5}  $

\item $\left(
\begin{array}{cc}
 -\frac{1}{2} \sqrt{-1+\sqrt{5}} & -\frac{1}{2} i \sqrt{5-\sqrt{5}} \\
 \frac{\sqrt{5-\sqrt{5}}}{2} & -\frac{1}{2} i \sqrt{-1+\sqrt{5}} \\
\end{array}
\right)\: \: \textrm{for} $

$F^{254}_{6}  $

\item $\left(
\begin{array}{cc}
 \frac{\sqrt{1+\sqrt{5}}}{2} & -\frac{i}{\sqrt{3+\sqrt{5}}} \\
 \frac{1}{\sqrt{3+\sqrt{5}}} & \frac{1}{2} i \sqrt{1+\sqrt{5}} \\
\end{array}
\right)\: \: \textrm{for} $

$F^{256}_{4}  $

\item $\left(
\begin{array}{cc}
 \frac{1}{4} \left(3-\sqrt{5}-i \sqrt{10-2 \sqrt{5}}\right) & \frac{1}{4} \left(i \sqrt{2 \left(1+\sqrt{5}\right)}+\sqrt{-10+6 \sqrt{5}}\right) \\
 -\frac{1}{4} i \left(\sqrt{2 \left(1+\sqrt{5}\right)}-i \sqrt{-10+6 \sqrt{5}}\right) & \frac{1}{2} \left(1-\sqrt{5}\right) \\
\end{array}
\right)\: \: \textrm{for} $

$F^{262}_{6}, F^{363}_{6}  $

\item $\left(
\begin{array}{cc}
 \frac{1}{4} \left(3-\sqrt{5}+i \sqrt{10-2 \sqrt{5}}\right) & \frac{1}{4} \left(-i \sqrt{2 \left(1+\sqrt{5}\right)}+\sqrt{-10+6 \sqrt{5}}\right) \\
 \frac{1}{4} i \left(\sqrt{2 \left(1+\sqrt{5}\right)}+i \sqrt{-10+6 \sqrt{5}}\right) & \frac{1}{2} \left(1-\sqrt{5}\right) \\
\end{array}
\right)\: \: \textrm{for} $

$F^{263}_{6}, F^{362}_{6}  $

\item $\left(
\begin{array}{cc}
 -\frac{1}{2} \sqrt{-1+\sqrt{5}} & -\frac{1}{2} \sqrt{5-\sqrt{5}} \\
 \frac{1}{2} i \sqrt{5-\sqrt{5}} & -\frac{1}{2} i \sqrt{-1+\sqrt{5}} \\
\end{array}
\right)\: \: \textrm{for} $

$F^{264}_{5}  $

\item $\left(
\begin{array}{cc}
 -\frac{1}{2} \sqrt{1+\sqrt{5}} & -\frac{\sqrt{2}}{1+\sqrt{5}} \\
 -\frac{i \left(-1+\sqrt{5}\right)}{2 \sqrt{2}} & \frac{1}{2} i \sqrt{1+\sqrt{5}} \\
\end{array}
\right)\: \: \textrm{for} $

$F^{266}_{3}  $

\item $\left(
\begin{array}{cc}
 -\frac{1}{2} \sqrt{1+\sqrt{5}} & \frac{-1+\sqrt{5}}{2 \sqrt{2}} \\
 \frac{\sqrt{2}}{1+\sqrt{5}} & \frac{\sqrt{1+\sqrt{5}}}{2} \\
\end{array}
\right)\: \: \textrm{for} $

$F^{326}_{6}, F^{425}_{6}  $

\item $\left(
\begin{array}{cc}
 \frac{1}{2} \left(1-\sqrt{5}\right) & \sqrt{\frac{1}{2} \left(-1+\sqrt{5}\right)} \\
 \sqrt{\frac{1}{2} \left(-1+\sqrt{5}\right)} & \sqrt{\frac{2}{3+\sqrt{5}}} \\
\end{array}
\right)\: \: \textrm{for} $

$F^{334}_{4}, F^{336}_{6}, F^{422}_{4}, F^{456}_{3}  $

\item $\left(
\begin{array}{cc}
 \frac{1}{2} \left(-1+\sqrt{5}\right) & \sqrt{\frac{1}{2} \left(-1+\sqrt{5}\right)} \\
 -\sqrt{\frac{1}{2} \left(-1+\sqrt{5}\right)} & \sqrt{\frac{2}{3+\sqrt{5}}} \\
\end{array}
\right)\: \: \textrm{for} $

$F^{344}_{3}, F^{424}_{4}, F^{434}_{3}, F^{455}_{4}  $

\item $\left(
\begin{array}{cc}
 -\frac{1}{2} \sqrt{-1+\sqrt{5}} & \frac{1}{2} i \sqrt{5-\sqrt{5}} \\
 \frac{\sqrt{5-\sqrt{5}}}{2} & \frac{1}{2} i \sqrt{-1+\sqrt{5}} \\
\end{array}
\right)\: \: \textrm{for} $

$F^{354}_{6}  $

\item $\left(
\begin{array}{cc}
 -\frac{1}{2} \sqrt{1+\sqrt{5}} & -\frac{i}{\sqrt{3+\sqrt{5}}} \\
 \frac{1}{\sqrt{3+\sqrt{5}}} & -\frac{1}{2} i \sqrt{1+\sqrt{5}} \\
\end{array}
\right)\: \: \textrm{for} $

$F^{356}_{4}  $

\item $\left(
\begin{array}{cc}
 -\frac{1}{2} \sqrt{-1+\sqrt{5}} & -\frac{1}{2} \sqrt{5-\sqrt{5}} \\
 -\frac{1}{2} i \sqrt{5-\sqrt{5}} & \frac{1}{2} i \sqrt{-1+\sqrt{5}} \\
\end{array}
\right)\: \: \textrm{for} $

$F^{364}_{5}  $

\item $\left(
\begin{array}{cc}
 -\frac{1}{2} \sqrt{1+\sqrt{5}} & -\frac{\sqrt{2}}{1+\sqrt{5}} \\
 \frac{i \left(-1+\sqrt{5}\right)}{2 \sqrt{2}} & -\frac{1}{2} i \sqrt{1+\sqrt{5}} \\
\end{array}
\right)\: \: \textrm{for} $

$F^{366}_{2}  $

\item $\left(
\begin{array}{cc}
 -\sqrt{\frac{2}{3+\sqrt{5}}} & -\sqrt{\frac{1}{2} \left(-1+\sqrt{5}\right)} \\
 -\sqrt{\frac{1}{2} \left(-1+\sqrt{5}\right)} & \sqrt{\frac{2}{3+\sqrt{5}}} \\
\end{array}
\right)\: \: \textrm{for} $

$F^{366}_{3}, F^{436}_{5}  $

\item $\left(
\begin{array}{cc}
 \frac{1}{2} \left(-1+\sqrt{5}\right) & -\sqrt{\frac{1}{2} \left(-1+\sqrt{5}\right)} \\
 \sqrt{\frac{1}{2} \left(-1+\sqrt{5}\right)} & \frac{1}{2} \left(-1+\sqrt{5}\right) \\
\end{array}
\right)\: \: \textrm{for} $

$F^{424}_{2}, F^{426}_{5}, F^{434}_{4}, F^{442}_{2}, F^{456}_{2}  $

\item $\left(
\begin{array}{cc}
 \frac{\sqrt{1+\sqrt{5}}}{2} & \frac{\sqrt{3-\sqrt{5}}}{2} \\
 \frac{1}{\sqrt{3+\sqrt{5}}} & -\frac{1}{2} \sqrt{1+\sqrt{5}} \\
\end{array}
\right)\: \: \textrm{for} $

$F^{435}_{6}  $

\item $\left(
\begin{array}{cc}
 \frac{1}{2} \sqrt{-1+\sqrt{5}} & \frac{\sqrt{5-\sqrt{5}}}{2} \\
 \frac{1}{2} i \sqrt{5-\sqrt{5}} & -\frac{1}{2} i \sqrt{-1+\sqrt{5}} \\
\end{array}
\right)\: \: \textrm{for} $

$F^{452}_{6}  $

\item $\left(
\begin{array}{cc}
 \frac{1}{2} \sqrt{-1+\sqrt{5}} & \frac{\sqrt{5-\sqrt{5}}}{2} \\
 -\frac{1}{2} i \sqrt{5-\sqrt{5}} & \frac{1}{2} i \sqrt{-1+\sqrt{5}} \\
\end{array}
\right)\: \: \textrm{for} $

$F^{453}_{6}  $

\item $\left(
\begin{array}{cc}
 \frac{1}{2} \left(-3+\sqrt{5}\right) & -\sqrt{\frac{1}{2} \left(-5+3 \sqrt{5}\right)} \\
 \sqrt{\frac{1}{2} \left(-5+3 \sqrt{5}\right)} & \frac{1}{2} \left(-3+\sqrt{5}\right) \\
\end{array}
\right)\: \: \textrm{for} $

$F^{454}_{5}  $

\item $\left(
\begin{array}{cc}
 \frac{1}{2} \sqrt{-1+\sqrt{5}} & -\frac{1}{2} i \sqrt{5-\sqrt{5}} \\
 -\frac{1}{2} \sqrt{5-\sqrt{5}} & -\frac{1}{2} i \sqrt{-1+\sqrt{5}} \\
\end{array}
\right)\: \: \textrm{for} $

$F^{462}_{5}  $

\item $\left(
\begin{array}{cc}
 \frac{1}{2} \sqrt{-1+\sqrt{5}} & \frac{1}{2} i \sqrt{5-\sqrt{5}} \\
 -\frac{1}{2} \sqrt{5-\sqrt{5}} & \frac{1}{2} i \sqrt{-1+\sqrt{5}} \\
\end{array}
\right)\: \: \textrm{for} $

$F^{463}_{5}  $

\item $\left(
\begin{array}{cc}
 -\frac{1}{2} \sqrt{1+\sqrt{5}} & -\frac{1}{\sqrt{3+\sqrt{5}}} \\
 -\frac{1}{\sqrt{3+\sqrt{5}}} & \frac{\sqrt{1+\sqrt{5}}}{2} \\
\end{array}
\right)\: \: \textrm{for} $

$F^{465}_{2}  $

\item $\left(
\begin{array}{cc}
 \frac{\sqrt{1+\sqrt{5}}}{2} & -\frac{1}{\sqrt{3+\sqrt{5}}} \\
 -\frac{1}{\sqrt{3+\sqrt{5}}} & -\frac{1}{2} \sqrt{1+\sqrt{5}} \\
\end{array}
\right)\: \: \textrm{for} $

$F^{465}_{3}  $

\item $\left(
\begin{array}{cc}
 -1 & 0 \\
 0 & i \\
\end{array}
\right)\: \: \textrm{for} $

$F^{466}_{1}  $

\item $\left(
\begin{array}{cc}
 \frac{1}{2} i \sqrt{1+\sqrt{5}} & -\frac{1}{2} i \sqrt{3-\sqrt{5}} \\
 \frac{i}{\sqrt{3+\sqrt{5}}} & \frac{1}{2} i \sqrt{1+\sqrt{5}} \\
\end{array}
\right)\: \: \textrm{for} $

$F^{524}_{6}  $

\item $\left(
\begin{array}{cc}
 -\frac{1}{2} i \sqrt{-1+\sqrt{5}} & -\frac{1}{2} i \sqrt{5-\sqrt{5}} \\
 \frac{1}{2} i \sqrt{5-\sqrt{5}} & -\frac{1}{2} i \sqrt{-1+\sqrt{5}} \\
\end{array}
\right)\: \: \textrm{for} $

$F^{526}_{4}  $

\item $\left(
\begin{array}{cc}
 -\frac{1}{2} i \sqrt{1+\sqrt{5}} & -\frac{1}{2} i \sqrt{3-\sqrt{5}} \\
 \frac{i}{\sqrt{3+\sqrt{5}}} & -\frac{1}{2} i \sqrt{1+\sqrt{5}} \\
\end{array}
\right)\: \: \textrm{for} $

$F^{534}_{6}  $

\item $\left(
\begin{array}{cc}
 -\frac{1}{2} i \sqrt{-1+\sqrt{5}} & \frac{1}{2} i \sqrt{5-\sqrt{5}} \\
 \frac{1}{2} i \sqrt{5-\sqrt{5}} & \frac{1}{2} i \sqrt{-1+\sqrt{5}} \\
\end{array}
\right)\: \: \textrm{for} $

$F^{536}_{4}  $

\item $\left(
\begin{array}{cc}
 -\frac{1}{2} i \left(-1+\sqrt{5}\right) & i \sqrt{\frac{1}{2} \left(-1+\sqrt{5}\right)} \\
 -i \sqrt{\frac{1}{2} \left(-1+\sqrt{5}\right)} & -\frac{1}{2} i \left(-1+\sqrt{5}\right) \\
\end{array}
\right)\: \: \textrm{for} $

$F^{542}_{6}, F^{654}_{2}  $

\item $\left(
\begin{array}{cc}
 \frac{1}{2} i \left(-1+\sqrt{5}\right) & -i \sqrt{\frac{1}{2} \left(-1+\sqrt{5}\right)} \\
 -i \sqrt{\frac{1}{2} \left(-1+\sqrt{5}\right)} & -\frac{1}{2} i \left(-1+\sqrt{5}\right) \\
\end{array}
\right)\: \: \textrm{for} $

$F^{543}_{6}, F^{554}_{4}  $

\item $\left(
\begin{array}{cc}
 \frac{1}{2} \left(-1+\sqrt{5}\right) & \sqrt{\frac{1}{2} \left(-1+\sqrt{5}\right)} \\
 i \sqrt{\frac{1}{2} \left(-1+\sqrt{5}\right)} & -\frac{1}{2} i \left(-1+\sqrt{5}\right) \\
\end{array}
\right)\: \: \textrm{for} $

$F^{544}_{5}  $

\item $\left(
\begin{array}{cc}
 -\frac{1}{2} i \left(-3+\sqrt{5}\right) & i \sqrt{\frac{1}{2} \left(-5+3 \sqrt{5}\right)} \\
 -i \sqrt{\frac{1}{2} \left(-5+3 \sqrt{5}\right)} & -\frac{1}{2} i \left(-3+\sqrt{5}\right) \\
\end{array}
\right)\: \: \textrm{for} $

$F^{545}_{4}  $

\item $\left(
\begin{array}{cc}
 -\frac{1}{2} i \sqrt{-1+\sqrt{5}} & -\frac{1}{2} i \sqrt{5-\sqrt{5}} \\
 \frac{\sqrt{5-\sqrt{5}}}{2} & -\frac{1}{2} \sqrt{-1+\sqrt{5}} \\
\end{array}
\right)\: \: \textrm{for} $

$F^{546}_{2}  $

\item $\left(
\begin{array}{cc}
 -\frac{1}{2} i \sqrt{-1+\sqrt{5}} & -\frac{1}{2} i \sqrt{5-\sqrt{5}} \\
 -\frac{1}{2} \sqrt{5-\sqrt{5}} & \frac{1}{2} \sqrt{-1+\sqrt{5}} \\
\end{array}
\right)\: \: \textrm{for} $

$F^{546}_{3}  $

\item $\left(
\begin{array}{cc}
 -\sqrt{\frac{1}{10} \left(5-\sqrt{5}\right)} & -\sqrt{\frac{1}{10} \left(5+\sqrt{5}\right)} \\
 -i \sqrt{\frac{1}{10} \left(5+\sqrt{5}\right)} & i \sqrt{\frac{1}{10} \left(5-\sqrt{5}\right)} \\
\end{array}
\right)\: \: \textrm{for} $

$F^{555}_{5}  $

\item $\left(
\begin{array}{cc}
 \frac{1}{2} i \left(-1+\sqrt{5}\right) & -i \sqrt{\frac{1}{2} \left(-1+\sqrt{5}\right)} \\
 i \sqrt{\frac{1}{2} \left(-1+\sqrt{5}\right)} & \frac{1}{2} i \left(-1+\sqrt{5}\right) \\
\end{array}
\right)\: \: \textrm{for} $

$F^{562}_{4}, F^{624}_{5}  $

\item $\left(
\begin{array}{cc}
 -\frac{1}{2} i \left(-1+\sqrt{5}\right) & -i \sqrt{\frac{1}{2} \left(-1+\sqrt{5}\right)} \\
 -i \sqrt{\frac{1}{2} \left(-1+\sqrt{5}\right)} & \frac{1}{2} i \left(-1+\sqrt{5}\right) \\
\end{array}
\right)\: \: \textrm{for} $

$F^{563}_{4}  $

\item $\left(
\begin{array}{cc}
 -\frac{1}{2} i \sqrt{1+\sqrt{5}} & \frac{i}{\sqrt{3+\sqrt{5}}} \\
 -\frac{i}{\sqrt{3+\sqrt{5}}} & -\frac{1}{2} i \sqrt{1+\sqrt{5}} \\
\end{array}
\right)\: \: \textrm{for} $

$F^{564}_{2}  $

\item $\left(
\begin{array}{cc}
 \frac{1}{2} i \sqrt{1+\sqrt{5}} & \frac{i}{\sqrt{3+\sqrt{5}}} \\
 -\frac{i}{\sqrt{3+\sqrt{5}}} & \frac{1}{2} i \sqrt{1+\sqrt{5}} \\
\end{array}
\right)\: \: \textrm{for} $

$F^{564}_{3}, F^{623}_{6}  $

\item $\left(
\begin{array}{cc}
 -\sqrt{\frac{2}{3+\sqrt{5}}} & -\sqrt{\frac{1}{2} \left(-1+\sqrt{5}\right)} \\
 -i \sqrt{\frac{1}{2} \left(-1+\sqrt{5}\right)} & i \sqrt{\frac{1}{2} \left(3-\sqrt{5}\right)} \\
\end{array}
\right)\: \: \textrm{for} $

$F^{622}_{6}  $

\item $\left(
\begin{array}{cc}
 \frac{1}{2} i \sqrt{-1+\sqrt{5}} & \frac{1}{2} i \sqrt{5-\sqrt{5}} \\
 \frac{1}{2} i \sqrt{5-\sqrt{5}} & -\frac{1}{2} i \sqrt{-1+\sqrt{5}} \\
\end{array}
\right)\: \: \textrm{for} $

$F^{625}_{4}  $

\item $\left(
\begin{array}{cc}
 \frac{1}{4} \left(\sqrt{10-2 \sqrt{5}}+i \left(-3+\sqrt{5}\right)\right) & -\frac{\sqrt{1+\sqrt{5}}+i \sqrt{-5+3 \sqrt{5}}}{2 \sqrt{2}} \\
 \frac{\sqrt{1+\sqrt{5}}+i \sqrt{-5+3 \sqrt{5}}}{2 \sqrt{2}} & \frac{1}{2} i \left(-1+\sqrt{5}\right) \\
\end{array}
\right)\: \: \textrm{for} $

$F^{626}_{2}, F^{636}_{3}  $

\item $\left(
\begin{array}{cc}
 \frac{1}{4} i \left(-3+\sqrt{5}+i \sqrt{10-2 \sqrt{5}}\right) & \frac{\sqrt{1+\sqrt{5}}-i \sqrt{-5+3 \sqrt{5}}}{2 \sqrt{2}} \\
 -\frac{\sqrt{1+\sqrt{5}}-i \sqrt{-5+3 \sqrt{5}}}{2 \sqrt{2}} & \frac{1}{2} i \left(-1+\sqrt{5}\right) \\
\end{array}
\right)\: \: \textrm{for} $

$F^{626}_{3}, F^{636}_{2}  $

\item $\left(
\begin{array}{cc}
 \frac{1}{2} i \sqrt{1+\sqrt{5}} & \frac{i \sqrt{2}}{1+\sqrt{5}} \\
 \frac{i \sqrt{2}}{1+\sqrt{5}} & -\frac{1}{2} i \sqrt{1+\sqrt{5}} \\
\end{array}
\right)\: \: \textrm{for} $

$F^{632}_{6}  $

\item $\left(
\begin{array}{cc}
 -\sqrt{\frac{2}{3+\sqrt{5}}} & -\sqrt{\frac{1}{2} \left(-1+\sqrt{5}\right)} \\
 i \sqrt{\frac{1}{2} \left(-1+\sqrt{5}\right)} & -i \sqrt{\frac{1}{2} \left(3-\sqrt{5}\right)} \\
\end{array}
\right)\: \: \textrm{for} $

$F^{633}_{6}  $

\item $\left(
\begin{array}{cc}
 -\frac{1}{2} i \left(-1+\sqrt{5}\right) & i \sqrt{\frac{1}{2} \left(-1+\sqrt{5}\right)} \\
 i \sqrt{\frac{1}{2} \left(-1+\sqrt{5}\right)} & \frac{1}{2} i \left(-1+\sqrt{5}\right) \\
\end{array}
\right)\: \: \textrm{for} $

$F^{634}_{5}  $

\item $\left(
\begin{array}{cc}
 \frac{1}{2} i \sqrt{-1+\sqrt{5}} & \frac{1}{2} i \sqrt{5-\sqrt{5}} \\
 -\frac{1}{2} i \sqrt{5-\sqrt{5}} & \frac{1}{2} i \sqrt{-1+\sqrt{5}} \\
\end{array}
\right)\: \: \textrm{for} $

$F^{635}_{4}  $

\item $\left(
\begin{array}{cc}
 -\frac{1}{2} i \sqrt{1+\sqrt{5}} & -\frac{1}{2} i \sqrt{3-\sqrt{5}} \\
 -\frac{i}{\sqrt{3+\sqrt{5}}} & \frac{1}{2} i \sqrt{1+\sqrt{5}} \\
\end{array}
\right)\: \: \textrm{for} $

$F^{642}_{5}  $

\item $\left(
\begin{array}{cc}
 \frac{1}{2} i \sqrt{1+\sqrt{5}} & -\frac{1}{2} i \sqrt{3-\sqrt{5}} \\
 -\frac{i}{\sqrt{3+\sqrt{5}}} & -\frac{1}{2} i \sqrt{1+\sqrt{5}} \\
\end{array}
\right)\: \: \textrm{for} $

$F^{643}_{5}  $

\item $\left(
\begin{array}{cc}
 \frac{1}{2} i \sqrt{-1+\sqrt{5}} & -\frac{1}{2} \sqrt{5-\sqrt{5}} \\
 -\frac{1}{2} i \sqrt{5-\sqrt{5}} & -\frac{1}{2} \sqrt{-1+\sqrt{5}} \\
\end{array}
\right)\: \: \textrm{for} $

$F^{645}_{2}  $

\item $\left(
\begin{array}{cc}
 \frac{1}{2} i \sqrt{-1+\sqrt{5}} & \frac{\sqrt{5-\sqrt{5}}}{2} \\
 -\frac{1}{2} i \sqrt{5-\sqrt{5}} & \frac{1}{2} \sqrt{-1+\sqrt{5}} \\
\end{array}
\right)\: \: \textrm{for} $

$F^{645}_{3}  $

\item $\left(
\begin{array}{cc}
 i & 0 \\
 0 & -i \\
\end{array}
\right)\: \: \textrm{for} $

$F^{646}_{1}  $

\item $\left(
\begin{array}{cc}
 \frac{1}{2} i \sqrt{1+\sqrt{5}} & -\frac{1}{2} i \sqrt{3-\sqrt{5}} \\
 \frac{1}{\sqrt{3+\sqrt{5}}} & \frac{\sqrt{1+\sqrt{5}}}{2} \\
\end{array}
\right)\: \: \textrm{for} $

$F^{652}_{4}  $

\item $\left(
\begin{array}{cc}
 -\frac{1}{2} i \sqrt{1+\sqrt{5}} & -\frac{1}{2} i \sqrt{3-\sqrt{5}} \\
 \frac{1}{\sqrt{3+\sqrt{5}}} & -\frac{1}{2} \sqrt{1+\sqrt{5}} \\
\end{array}
\right)\: \: \textrm{for} $

$F^{653}_{4}  $

\item $\left(
\begin{array}{cc}
 \frac{1}{2} i \left(-1+\sqrt{5}\right) & i \sqrt{\frac{1}{2} \left(-1+\sqrt{5}\right)} \\
 i \sqrt{\frac{1}{2} \left(-1+\sqrt{5}\right)} & -\frac{1}{2} i \left(-1+\sqrt{5}\right) \\
\end{array}
\right)\: \: \textrm{for} $

$F^{654}_{3}  $

\item $\left(
\begin{array}{cc}
 -i \sqrt{\frac{2}{3+\sqrt{5}}} & -i \sqrt{\frac{1}{2} \left(-1+\sqrt{5}\right)} \\
 i \sqrt{\frac{1}{2} \left(-1+\sqrt{5}\right)} & -i \sqrt{\frac{2}{3+\sqrt{5}}} \\
\end{array}
\right)\: \: \textrm{for} $

$F^{662}_{2}  $

\item $\left(
\begin{array}{cc}
 -\frac{1}{2} i \sqrt{1+\sqrt{5}} & \frac{\sqrt{2}}{1+\sqrt{5}} \\
 \frac{i \sqrt{2}}{1+\sqrt{5}} & \frac{\sqrt{1+\sqrt{5}}}{2} \\
\end{array}
\right)\: \: \textrm{for} $

$F^{662}_{3}  $

\item $\left(
\begin{array}{cc}
 -\frac{1}{2} i \sqrt{1+\sqrt{5}} & -\frac{\sqrt{2}}{1+\sqrt{5}} \\
 \frac{i \sqrt{2}}{1+\sqrt{5}} & -\frac{1}{2} \sqrt{1+\sqrt{5}} \\
\end{array}
\right)\: \: \textrm{for} $

$F^{663}_{2}  $

\item $\left(
\begin{array}{cc}
 i \sqrt{\frac{2}{3+\sqrt{5}}} & i \sqrt{\frac{1}{2} \left(-1+\sqrt{5}\right)} \\
 -i \sqrt{\frac{1}{2} \left(-1+\sqrt{5}\right)} & i \sqrt{\frac{2}{3+\sqrt{5}}} \\
\end{array}
\right)\: \: \textrm{for} $

$F^{663}_{3}  $

\item $\left(
\begin{array}{cc}
 i & 0 \\
 0 & 1 \\
\end{array}
\right)\: \: \textrm{for} $

$F^{664}_{1}  $

\item $\left(
\begin{array}{ccc}
 \frac{1}{2} \left(-1+\sqrt{5}\right) & \sqrt{\frac{1}{2} \left(-2+\sqrt{5}\right)} & -\frac{1}{\sqrt{2}} \\
 \sqrt{\frac{1}{2} \left(-2+\sqrt{5}\right)} & \frac{1}{4} \left(5-\sqrt{5}\right) & \frac{1}{2} \sqrt{\frac{1}{2} \left(1+\sqrt{5}\right)} \\
 -\frac{1}{\sqrt{2}} & \frac{1}{2} \sqrt{\frac{1}{2} \left(1+\sqrt{5}\right)} & \frac{1}{4} \left(1-\sqrt{5}\right) \\
\end{array}
\right)\: \: \textrm{for} $

$F^{246}_{6}  $

\item $\left(
\begin{array}{ccc}
 \frac{1}{2} \left(1-\sqrt{5}\right) & \frac{1}{2} \sqrt{-5+3 \sqrt{5}} & -\frac{i}{\sqrt{3+\sqrt{5}}} \\
 -\frac{1}{2} \sqrt{-5+3 \sqrt{5}} & 1-\frac{\sqrt{5}}{2} & \frac{i \sqrt[4]{5}}{2} \\
 -\frac{i}{\sqrt{3+\sqrt{5}}} & -\frac{1}{2} i \sqrt[4]{5} & \frac{1}{2} \\
\end{array}
\right)\: \: \textrm{for} $

$F^{264}_{6}, F^{463}_{6}  $

\item $\left(
\begin{array}{ccc}
 \frac{1}{2} \left(-1+\sqrt{5}\right) & -\sqrt{\frac{1}{2} \left(-2+\sqrt{5}\right)} & -\frac{i}{\sqrt{2}} \\
 -\sqrt{\frac{1}{2} \left(-2+\sqrt{5}\right)} & \frac{1}{4} \left(5-\sqrt{5}\right) & -\frac{1}{2} i \sqrt{\frac{1}{2} \left(1+\sqrt{5}\right)} \\
 \frac{i}{\sqrt{2}} & \frac{1}{2} i \sqrt{\frac{1}{2} \left(1+\sqrt{5}\right)} & \frac{1}{4} \left(1-\sqrt{5}\right) \\
\end{array}
\right)\: \: \textrm{for} $

$F^{266}_{4}  $

\item $\left(
\begin{array}{ccc}
 \frac{1}{2} \left(1-\sqrt{5}\right) & \sqrt{\frac{1}{2} \left(-2+\sqrt{5}\right)} & \frac{1}{\sqrt{2}} \\
 -\sqrt{\frac{1}{2} \left(-2+\sqrt{5}\right)} & \frac{1}{4} \left(5-\sqrt{5}\right) & -\frac{1}{2} \sqrt{\frac{1}{2} \left(1+\sqrt{5}\right)} \\
 -\frac{1}{\sqrt{2}} & -\frac{1}{2} \sqrt{\frac{1}{2} \left(1+\sqrt{5}\right)} & \frac{1}{4} \left(1-\sqrt{5}\right) \\
\end{array}
\right)\: \: \textrm{for} $

$F^{346}_{6}  $

\item $\left(
\begin{array}{ccc}
 \frac{1}{2} \left(1-\sqrt{5}\right) & \frac{1}{2} \sqrt{-5+3 \sqrt{5}} & \frac{i}{\sqrt{3+\sqrt{5}}} \\
 -\frac{1}{2} \sqrt{-5+3 \sqrt{5}} & 1-\frac{\sqrt{5}}{2} & -\frac{1}{2} i \sqrt[4]{5} \\
 \frac{i}{\sqrt{3+\sqrt{5}}} & \frac{i \sqrt[4]{5}}{2} & \frac{1}{2} \\
\end{array}
\right)\: \: \textrm{for} $

$F^{364}_{6}, F^{462}_{6}  $

\item $\left(
\begin{array}{ccc}
 \frac{1}{2} \left(1-\sqrt{5}\right) & \sqrt{\frac{1}{2} \left(-2+\sqrt{5}\right)} & -\frac{i}{\sqrt{2}} \\
 -\sqrt{\frac{1}{2} \left(-2+\sqrt{5}\right)} & \frac{1}{4} \left(5-\sqrt{5}\right) & \frac{1}{2} i \sqrt{\frac{1}{2} \left(1+\sqrt{5}\right)} \\
 -\frac{i}{\sqrt{2}} & -\frac{1}{2} i \sqrt{\frac{1}{2} \left(1+\sqrt{5}\right)} & \frac{1}{4} \left(1-\sqrt{5}\right) \\
\end{array}
\right)\: \: \textrm{for} $

$F^{366}_{4}  $

\item $\left(
\begin{array}{ccc}
 \frac{1}{2} \left(1-\sqrt{5}\right) & -\sqrt{\frac{1}{2} \left(-2+\sqrt{5}\right)} & -\frac{1}{\sqrt{2}} \\
 -\sqrt{\frac{1}{2} \left(-2+\sqrt{5}\right)} & \frac{1}{4} \left(-5+\sqrt{5}\right) & \frac{1}{2} \sqrt{\frac{1}{2} \left(1+\sqrt{5}\right)} \\
 -\frac{1}{\sqrt{2}} & \frac{1}{2} \sqrt{\frac{1}{2} \left(1+\sqrt{5}\right)} & \frac{1}{4} \left(-1+\sqrt{5}\right) \\
\end{array}
\right)\: \: \textrm{for} $

$F^{426}_{6}  $

\item $\left(
\begin{array}{ccc}
 \frac{1}{2} \left(-1+\sqrt{5}\right) & -\sqrt{\frac{1}{2} \left(-2+\sqrt{5}\right)} & \frac{1}{\sqrt{2}} \\
 \sqrt{\frac{1}{2} \left(-2+\sqrt{5}\right)} & \frac{1}{4} \left(-5+\sqrt{5}\right) & -\frac{1}{2} \sqrt{\frac{1}{2} \left(1+\sqrt{5}\right)} \\
 -\frac{1}{\sqrt{2}} & -\frac{1}{2} \sqrt{\frac{1}{2} \left(1+\sqrt{5}\right)} & \frac{1}{4} \left(-1+\sqrt{5}\right) \\
\end{array}
\right)\: \: \textrm{for} $

$F^{436}_{6}  $

\item $\left(
\begin{array}{ccc}
 \sqrt{\frac{1}{2} \left(3-\sqrt{5}\right)} & -\sqrt{\frac{1}{2} \left(3-\sqrt{5}\right)} & \sqrt{-2+\sqrt{5}} \\
 -\sqrt{\frac{1}{2} \left(-2+\sqrt{5}\right)} & \sqrt{\frac{1}{2} \left(-2+\sqrt{5}\right)} & \sqrt{3-\sqrt{5}} \\
 \frac{1}{\sqrt{2}} & \frac{1}{\sqrt{2}} & 0 \\
\end{array}
\right)\: \: \textrm{for} $

$F^{445}_{6}  $

\item $\left(
\begin{array}{ccc}
 \sqrt{\frac{1}{2} \left(3-\sqrt{5}\right)} & -\sqrt{\frac{1}{2} \left(3-\sqrt{5}\right)} & \sqrt{-2+\sqrt{5}} \\
 -\sqrt{\frac{1}{2} \left(-2+\sqrt{5}\right)} & \sqrt{\frac{1}{2} \left(-2+\sqrt{5}\right)} & \sqrt{3-\sqrt{5}} \\
 -\frac{1}{\sqrt{2}} & -\frac{1}{\sqrt{2}} & 0 \\
\end{array}
\right)\: \: \textrm{for} $

$F^{446}_{5}  $

\item $\left(
\begin{array}{ccc}
 -\sqrt{5-2 \sqrt{5}} & \sqrt{2 \left(-2+\sqrt{5}\right)} & 0 \\
 -\sqrt{2 \left(-2+\sqrt{5}\right)} & -\sqrt{5-2 \sqrt{5}} & 0 \\
 0 & 0 & i \\
\end{array}
\right)\: \: \textrm{for} $

$F^{454}_{6}  $

\item $\left(
\begin{array}{ccc}
 -\sqrt{\frac{1}{2} \left(3-\sqrt{5}\right)} & -\sqrt{\frac{1}{2} \left(-2+\sqrt{5}\right)} & \frac{i}{\sqrt{2}} \\
 \sqrt{\frac{1}{2} \left(3-\sqrt{5}\right)} & \sqrt{\frac{1}{2} \left(-2+\sqrt{5}\right)} & \frac{i}{\sqrt{2}} \\
 \sqrt{-2+\sqrt{5}} & -\sqrt{3-\sqrt{5}} & 0 \\
\end{array}
\right)\: \: \textrm{for} $

$F^{456}_{4}  $

\item $\left(
\begin{array}{ccc}
 \sqrt{5-2 \sqrt{5}} & -\sqrt{2 \left(-2+\sqrt{5}\right)} & 0 \\
 \sqrt{2 \left(-2+\sqrt{5}\right)} & \sqrt{5-2 \sqrt{5}} & 0 \\
 0 & 0 & -i \\
\end{array}
\right)\: \: \textrm{for} $

$F^{464}_{5}  $

\item $\left(
\begin{array}{ccc}
 \sqrt{\frac{1}{2} \left(3-\sqrt{5}\right)} & \sqrt{\frac{1}{2} \left(-2+\sqrt{5}\right)} & \frac{1}{\sqrt{2}} \\
 -\sqrt{\frac{1}{2} \left(3-\sqrt{5}\right)} & -\sqrt{\frac{1}{2} \left(-2+\sqrt{5}\right)} & \frac{1}{\sqrt{2}} \\
 -\sqrt{-2+\sqrt{5}} & \sqrt{3-\sqrt{5}} & 0 \\
\end{array}
\right)\: \: \textrm{for} $

$F^{465}_{4}  $

\item $\left(
\begin{array}{ccc}
 \frac{1}{2} \left(-1+\sqrt{5}\right) & -\sqrt{\frac{1}{2} \left(-2+\sqrt{5}\right)} & \frac{1}{\sqrt{2}} \\
 -\sqrt{\frac{1}{2} \left(-2+\sqrt{5}\right)} & \frac{1}{4} \left(5-\sqrt{5}\right) & \frac{1}{2} \sqrt{\frac{1}{2} \left(1+\sqrt{5}\right)} \\
 -\frac{i}{\sqrt{2}} & -\frac{1}{2} i \sqrt{\frac{1}{2} \left(1+\sqrt{5}\right)} & \frac{1}{4} i \left(-1+\sqrt{5}\right) \\
\end{array}
\right)\: \: \textrm{for} $

$F^{466}_{2}  $

\item $\left(
\begin{array}{ccc}
 \frac{1}{2} \left(1-\sqrt{5}\right) & \sqrt{\frac{1}{2} \left(-2+\sqrt{5}\right)} & \frac{1}{\sqrt{2}} \\
 -\sqrt{\frac{1}{2} \left(-2+\sqrt{5}\right)} & \frac{1}{4} \left(5-\sqrt{5}\right) & -\frac{1}{2} \sqrt{\frac{1}{2} \left(1+\sqrt{5}\right)} \\
 \frac{i}{\sqrt{2}} & \frac{1}{2} i \sqrt{\frac{1}{2} \left(1+\sqrt{5}\right)} & \frac{1}{4} i \left(-1+\sqrt{5}\right) \\
\end{array}
\right)\: \: \textrm{for} $

$F^{466}_{3}  $

\item $\left(
\begin{array}{ccc}
 i \sqrt{\frac{1}{2} \left(3-\sqrt{5}\right)} & i \sqrt{\frac{1}{2} \left(-2+\sqrt{5}\right)} & -\frac{i}{\sqrt{2}} \\
 -i \sqrt{\frac{1}{2} \left(3-\sqrt{5}\right)} & -i \sqrt{\frac{1}{2} \left(-2+\sqrt{5}\right)} & -\frac{i}{\sqrt{2}} \\
 -i \sqrt{-2+\sqrt{5}} & i \sqrt{3-\sqrt{5}} & 0 \\
\end{array}
\right)\: \: \textrm{for} $

$F^{544}_{6}  $

\item $\left(
\begin{array}{ccc}
 i \sqrt{5-2 \sqrt{5}} & -i \sqrt{2 \left(-2+\sqrt{5}\right)} & 0 \\
 i \sqrt{2 \left(-2+\sqrt{5}\right)} & i \sqrt{5-2 \sqrt{5}} & 0 \\
 0 & 0 & i \\
\end{array}
\right)\: \: \textrm{for} $

$F^{546}_{4}, F^{645}_{4}  $

\item $\left(
\begin{array}{ccc}
 -i \sqrt{\frac{1}{10} \left(5-\sqrt{5}\right)} & -\frac{i}{\sqrt{2} \sqrt[4]{5}} & -\frac{1}{\sqrt{2}} \\
 i \sqrt{\frac{1}{10} \left(5-\sqrt{5}\right)} & \frac{i}{\sqrt{2} \sqrt[4]{5}} & -\frac{1}{\sqrt{2}} \\
 \frac{i}{\sqrt[4]{5}} & -i \sqrt{1-\frac{1}{\sqrt{5}}} & 0 \\
\end{array}
\right)\: \: \textrm{for} $

$F^{556}_{6}  $

\item $\left(
\begin{array}{ccc}
 -i \sqrt{\frac{1}{2} \left(3-\sqrt{5}\right)} & i \sqrt{\frac{1}{2} \left(3-\sqrt{5}\right)} & -i \sqrt{-2+\sqrt{5}} \\
 i \sqrt{\frac{1}{2} \left(-2+\sqrt{5}\right)} & -i \sqrt{\frac{1}{2} \left(-2+\sqrt{5}\right)} & -i \sqrt{3-\sqrt{5}} \\
 \frac{i}{\sqrt{2}} & \frac{i}{\sqrt{2}} & 0 \\
\end{array}
\right)\: \: \textrm{for} $

$F^{564}_{4}  $

\item $\left(
\begin{array}{ccc}
 -\frac{1}{2}+\frac{1}{2} i \sqrt{1-\frac{2}{\sqrt{5}}} & -\frac{1}{2}-\frac{1}{2} i \sqrt{1-\frac{2}{\sqrt{5}}} & -\frac{i}{\sqrt[4]{5}} \\
 -\frac{1}{2}-\frac{1}{2} i \sqrt{1-\frac{2}{\sqrt{5}}} & -\frac{1}{2}+\frac{1}{2} i \sqrt{1-\frac{2}{\sqrt{5}}} & \frac{i}{\sqrt[4]{5}} \\
 \frac{i}{\sqrt[4]{5}} & -\frac{i}{\sqrt[4]{5}} & i \sqrt{1-\frac{2}{\sqrt{5}}} \\
\end{array}
\right)\: \: \textrm{for} $

$F^{565}_{6}  $

\item $\left(
\begin{array}{ccc}
 -\sqrt{\frac{1}{10} \left(5-\sqrt{5}\right)} & \sqrt{\frac{1}{10} \left(5-\sqrt{5}\right)} & -\frac{1}{\sqrt[4]{5}} \\
 -\frac{i}{\sqrt{2} \sqrt[4]{5}} & \frac{i}{\sqrt{2} \sqrt[4]{5}} & i \sqrt{1-\frac{1}{\sqrt{5}}} \\
 -\frac{i}{\sqrt{2}} & -\frac{i}{\sqrt{2}} & 0 \\
\end{array}
\right)\: \: \textrm{for} $

$F^{566}_{5}  $

\item $\left(
\begin{array}{ccc}
 -\frac{1}{2} i \left(-1+\sqrt{5}\right) & i \sqrt{\frac{1}{2} \left(-2+\sqrt{5}\right)} & \frac{i}{\sqrt{2}} \\
 i \sqrt{\frac{1}{2} \left(-2+\sqrt{5}\right)} & \frac{1}{4} i \left(-5+\sqrt{5}\right) & \frac{1}{2} i \sqrt{\frac{1}{2} \left(1+\sqrt{5}\right)} \\
 \frac{i}{\sqrt{2}} & \frac{1}{2} i \sqrt{\frac{1}{2} \left(1+\sqrt{5}\right)} & \frac{i}{1+\sqrt{5}} \\
\end{array}
\right)\: \: \textrm{for} $

$F^{624}_{6}  $

\item $\left(
\begin{array}{ccc}
 \frac{1}{2} i \left(-1+\sqrt{5}\right) & -\frac{1}{2} i \sqrt{-5+3 \sqrt{5}} & \frac{i}{\sqrt{3+\sqrt{5}}} \\
 \frac{1}{2} i \sqrt{-5+3 \sqrt{5}} & \frac{1}{2} i \left(-2+\sqrt{5}\right) & -\frac{1}{2} i \sqrt[4]{5} \\
 -\frac{i}{\sqrt{3+\sqrt{5}}} & -\frac{1}{2} i \sqrt[4]{5} & -\frac{i}{2} \\
\end{array}
\right)\: \: \textrm{for} $

$F^{626}_{4}  $

\item $\left(
\begin{array}{ccc}
 \frac{1}{2} i \left(-1+\sqrt{5}\right) & -i \sqrt{\frac{1}{2} \left(-2+\sqrt{5}\right)} & \frac{i}{\sqrt{2}} \\
 i \sqrt{\frac{1}{2} \left(-2+\sqrt{5}\right)} & \frac{1}{4} i \left(-5+\sqrt{5}\right) & -\frac{1}{2} i \sqrt{\frac{1}{2} \left(1+\sqrt{5}\right)} \\
 -\frac{i}{\sqrt{2}} & -\frac{1}{2} i \sqrt{\frac{1}{2} \left(1+\sqrt{5}\right)} & \frac{i}{1+\sqrt{5}} \\
\end{array}
\right)\: \: \textrm{for} $

$F^{634}_{6}  $

\item $\left(
\begin{array}{ccc}
 \frac{1}{2} i \left(-1+\sqrt{5}\right) & -\frac{1}{2} i \sqrt{-5+3 \sqrt{5}} & -\frac{i}{\sqrt{3+\sqrt{5}}} \\
 \frac{1}{2} i \sqrt{-5+3 \sqrt{5}} & \frac{1}{2} i \left(-2+\sqrt{5}\right) & \frac{i \sqrt[4]{5}}{2} \\
 \frac{i}{\sqrt{3+\sqrt{5}}} & \frac{i \sqrt[4]{5}}{2} & -\frac{i}{2} \\
\end{array}
\right)\: \: \textrm{for} $

$F^{636}_{4}  $

\item $\left(
\begin{array}{ccc}
 \frac{1}{2} i \left(-1+\sqrt{5}\right) & -i \sqrt{\frac{1}{2} \left(-2+\sqrt{5}\right)} & \frac{i}{\sqrt{2}} \\
 -i \sqrt{\frac{1}{2} \left(-2+\sqrt{5}\right)} & -\frac{1}{4} i \left(-5+\sqrt{5}\right) & \frac{1}{2} i \sqrt{\frac{1}{2} \left(1+\sqrt{5}\right)} \\
 \frac{i}{\sqrt{2}} & \frac{1}{2} i \sqrt{\frac{1}{2} \left(1+\sqrt{5}\right)} & -\frac{1}{4} i \left(-1+\sqrt{5}\right) \\
\end{array}
\right)\: \: \textrm{for} $

$F^{642}_{6}  $

\item $\left(
\begin{array}{ccc}
 -\frac{1}{2} i \left(-1+\sqrt{5}\right) & i \sqrt{\frac{1}{2} \left(-2+\sqrt{5}\right)} & \frac{i}{\sqrt{2}} \\
 -i \sqrt{\frac{1}{2} \left(-2+\sqrt{5}\right)} & -\frac{1}{4} i \left(-5+\sqrt{5}\right) & -\frac{1}{2} i \sqrt{\frac{1}{2} \left(1+\sqrt{5}\right)} \\
 -\frac{i}{\sqrt{2}} & -\frac{1}{2} i \sqrt{\frac{1}{2} \left(1+\sqrt{5}\right)} & -\frac{1}{4} i \left(-1+\sqrt{5}\right) \\
\end{array}
\right)\: \: \textrm{for} $

$F^{643}_{6}  $

\item $\left(
\begin{array}{ccc}
 i \sqrt{\frac{1}{2} \left(3-\sqrt{5}\right)} & i \sqrt{\frac{1}{2} \left(-2+\sqrt{5}\right)} & \frac{i}{\sqrt{2}} \\
 -i \sqrt{\frac{1}{2} \left(3-\sqrt{5}\right)} & -i \sqrt{\frac{1}{2} \left(-2+\sqrt{5}\right)} & \frac{i}{\sqrt{2}} \\
 -i \sqrt{-2+\sqrt{5}} & i \sqrt{3-\sqrt{5}} & 0 \\
\end{array}
\right)\: \: \textrm{for} $

$F^{644}_{5}  $

\item $\left(
\begin{array}{ccc}
 -\frac{1}{2} i \left(-1+\sqrt{5}\right) & \frac{1}{2} i \sqrt{-5+3 \sqrt{5}} & \frac{1}{\sqrt{3+\sqrt{5}}} \\
 -\frac{1}{2} i \sqrt{-5+3 \sqrt{5}} & -\frac{1}{2} i \left(-2+\sqrt{5}\right) & -\frac{\sqrt[4]{5}}{2} \\
 -\frac{1}{\sqrt{3+\sqrt{5}}} & -\frac{\sqrt[4]{5}}{2} & -\frac{i}{2} \\
\end{array}
\right)\: \: \textrm{for} $

$F^{646}_{2}  $

\item $\left(
\begin{array}{ccc}
 -\frac{1}{2} i \left(-1+\sqrt{5}\right) & \frac{1}{2} i \sqrt{-5+3 \sqrt{5}} & -\frac{1}{\sqrt{3+\sqrt{5}}} \\
 -\frac{1}{2} i \sqrt{-5+3 \sqrt{5}} & -\frac{1}{2} i \left(-2+\sqrt{5}\right) & \frac{\sqrt[4]{5}}{2} \\
 \frac{1}{\sqrt{3+\sqrt{5}}} & \frac{\sqrt[4]{5}}{2} & -\frac{i}{2} \\
\end{array}
\right)\: \: \textrm{for} $

$F^{646}_{3}  $

\item $\left(
\begin{array}{ccc}
 i \sqrt{\frac{1}{2} \left(3-\sqrt{5}\right)} & -i \sqrt{\frac{1}{2} \left(3-\sqrt{5}\right)} & i \sqrt{-2+\sqrt{5}} \\
 -i \sqrt{\frac{1}{2} \left(-2+\sqrt{5}\right)} & i \sqrt{\frac{1}{2} \left(-2+\sqrt{5}\right)} & i \sqrt{3-\sqrt{5}} \\
 -\frac{1}{\sqrt{2}} & -\frac{1}{\sqrt{2}} & 0 \\
\end{array}
\right)\: \: \textrm{for} $

$F^{654}_{4}  $

\item $\left(
\begin{array}{ccc}
 -\sqrt{\frac{1}{10} \left(5-\sqrt{5}\right)} & \sqrt{\frac{1}{10} \left(5-\sqrt{5}\right)} & -\frac{1}{\sqrt[4]{5}} \\
 \frac{i}{\sqrt{2} \sqrt[4]{5}} & -\frac{i}{\sqrt{2} \sqrt[4]{5}} & -i \sqrt{1-\frac{1}{\sqrt{5}}} \\
 \frac{1}{\sqrt{2}} & \frac{1}{\sqrt{2}} & 0 \\
\end{array}
\right)\: \: \textrm{for} $

$F^{655}_{6}  $

\item $\left(
\begin{array}{ccc}
 -\frac{1}{2}-\frac{1}{2} i \sqrt{1-\frac{2}{\sqrt{5}}} & -\frac{1}{2}+\frac{1}{2} i \sqrt{1-\frac{2}{\sqrt{5}}} & \frac{i}{\sqrt[4]{5}} \\
 -\frac{1}{2}+\frac{1}{2} i \sqrt{1-\frac{2}{\sqrt{5}}} & -\frac{1}{2}-\frac{1}{2} i \sqrt{1-\frac{2}{\sqrt{5}}} & -\frac{i}{\sqrt[4]{5}} \\
 -\frac{i}{\sqrt[4]{5}} & \frac{i}{\sqrt[4]{5}} & -i \sqrt{1-\frac{2}{\sqrt{5}}} \\
\end{array}
\right)\: \: \textrm{for} $

$F^{656}_{5}  $

\item $\left(
\begin{array}{ccc}
 -\frac{1}{2} i \left(-1+\sqrt{5}\right) & -i \sqrt{\frac{1}{2} \left(-2+\sqrt{5}\right)} & -\frac{1}{\sqrt{2}} \\
 -i \sqrt{\frac{1}{2} \left(-2+\sqrt{5}\right)} & \frac{1}{4} i \left(-5+\sqrt{5}\right) & \frac{1}{2} \sqrt{\frac{1}{2} \left(1+\sqrt{5}\right)} \\
 \frac{1}{\sqrt{2}} & -\frac{1}{2} \sqrt{\frac{1}{2} \left(1+\sqrt{5}\right)} & \frac{1}{4} i \left(-1+\sqrt{5}\right) \\
\end{array}
\right)\: \: \textrm{for} $

$F^{662}_{4}  $

\item $\left(
\begin{array}{ccc}
 \frac{1}{2} i \left(-1+\sqrt{5}\right) & -i \sqrt{\frac{1}{2} \left(-2+\sqrt{5}\right)} & \frac{1}{\sqrt{2}} \\
 i \sqrt{\frac{1}{2} \left(-2+\sqrt{5}\right)} & \frac{1}{4} i \left(-5+\sqrt{5}\right) & -\frac{1}{2} \sqrt{\frac{1}{2} \left(1+\sqrt{5}\right)} \\
 \frac{1}{\sqrt{2}} & \frac{1}{2} \sqrt{\frac{1}{2} \left(1+\sqrt{5}\right)} & \frac{1}{4} i \left(-1+\sqrt{5}\right) \\
\end{array}
\right)\: \: \textrm{for} $

$F^{663}_{4}  $

\item $\left(
\begin{array}{ccc}
 -\frac{1}{2} i \left(-1+\sqrt{5}\right) & -i \sqrt{\frac{1}{2} \left(-2+\sqrt{5}\right)} & \frac{1}{\sqrt{2}} \\
 -i \sqrt{\frac{1}{2} \left(-2+\sqrt{5}\right)} & \frac{1}{4} i \left(-5+\sqrt{5}\right) & -\frac{1}{2} \sqrt{\frac{1}{2} \left(1+\sqrt{5}\right)} \\
 \frac{i}{\sqrt{2}} & -\frac{1}{2} i \sqrt{\frac{1}{2} \left(1+\sqrt{5}\right)} & \frac{1}{4} \left(-1+\sqrt{5}\right) \\
\end{array}
\right)\: \: \textrm{for} $

$F^{664}_{2}  $

\item $\left(
\begin{array}{ccc}
 \frac{1}{2} i \left(-1+\sqrt{5}\right) & -i \sqrt{\frac{1}{2} \left(-2+\sqrt{5}\right)} & -\frac{1}{\sqrt{2}} \\
 i \sqrt{\frac{1}{2} \left(-2+\sqrt{5}\right)} & \frac{1}{4} i \left(-5+\sqrt{5}\right) & \frac{1}{2} \sqrt{\frac{1}{2} \left(1+\sqrt{5}\right)} \\
 \frac{i}{\sqrt{2}} & \frac{1}{2} i \sqrt{\frac{1}{2} \left(1+\sqrt{5}\right)} & \frac{1}{4} \left(-1+\sqrt{5}\right) \\
\end{array}
\right)\: \: \textrm{for} $

$F^{664}_{3}  $

\item $\left(
\begin{array}{ccc}
 i \sqrt{\frac{1}{10} \left(5-\sqrt{5}\right)} & \frac{i}{\sqrt{2} \sqrt[4]{5}} & \frac{i}{\sqrt{2}} \\
 -i \sqrt{\frac{1}{10} \left(5-\sqrt{5}\right)} & -\frac{i}{\sqrt{2} \sqrt[4]{5}} & \frac{i}{\sqrt{2}} \\
 -\frac{i}{\sqrt[4]{5}} & i \sqrt{1-\frac{1}{\sqrt{5}}} & 0 \\
\end{array}
\right)\: \: \textrm{for} $

$F^{665}_{5}  $

\item $\left(
\begin{array}{cccc}
 \frac{1}{2} \left(3-\sqrt{5}\right) & -\sqrt{-2+\sqrt{5}} & \sqrt{-2+\sqrt{5}} & \frac{1}{2} \left(1-\sqrt{5}\right) \\
 \sqrt{-2+\sqrt{5}} & \frac{1}{2} \left(3-\sqrt{5}\right) & \frac{1}{2} \left(-1+\sqrt{5}\right) & \sqrt{-2+\sqrt{5}} \\
 -\sqrt{-2+\sqrt{5}} & \frac{1}{2} \left(-1+\sqrt{5}\right) & \frac{1}{2} \left(3-\sqrt{5}\right) & -\sqrt{-2+\sqrt{5}} \\
 \frac{1}{2} \left(1-\sqrt{5}\right) & -\sqrt{-2+\sqrt{5}} & \sqrt{-2+\sqrt{5}} & \frac{1}{2} \left(3-\sqrt{5}\right) \\
\end{array}
\right)\: \: \textrm{for} $

$F^{444}_{4}  $

{\tiny{
\item $\left(
\begin{array}{cccc}
 a & \bar{a} & \frac{1}{2} i \sqrt{-1+\sqrt{5}} & -\frac{1}{2} \sqrt{1-\frac{1}{\sqrt{5}}} \\
 \bar{a} & a & -\frac{1}{2} i \sqrt{-1+\sqrt{5}} & -\frac{1}{2} \sqrt{1-\frac{1}{\sqrt{5}}} \\
 -\frac{1}{2} i \sqrt{-1+\sqrt{5}} & \frac{1}{2} i \sqrt{-1+\sqrt{5}} & \frac{1}{2} i \left(-1+\sqrt{5}\right) & 0 \\
 \frac{1}{2} i \sqrt{1-\frac{1}{\sqrt{5}}} & \frac{1}{2} i \sqrt{1-\frac{1}{\sqrt{5}}} & 0 & i \sqrt{\frac{1}{10} \left(5+\sqrt{5}\right)} \\
\end{array}
\right)\: \: \textrm{for} $

$F^{566}_{6}  $

\item $\left(
\begin{array}{cccc}
 \bar{a} & a & -\frac{1}{2} i \sqrt{-1+\sqrt{5}} & \frac{1}{2} \sqrt{1-\frac{1}{\sqrt{5}}} \\
 a & \bar{a} & \frac{1}{2} i \sqrt{-1+\sqrt{5}} & \frac{1}{2} \sqrt{1-\frac{1}{\sqrt{5}}} \\
 \frac{1}{2} i \sqrt{-1+\sqrt{5}} & -\frac{1}{2} i \sqrt{-1+\sqrt{5}} & -\frac{1}{2} i \left(-1+\sqrt{5}\right) & 0 \\
 -\frac{1}{2} \sqrt{1-\frac{1}{\sqrt{5}}} & -\frac{1}{2} \sqrt{1-\frac{1}{\sqrt{5}}} & 0 & \sqrt{\frac{1}{10} \left(5+\sqrt{5}\right)} \\
\end{array}
\right)\: \: \textrm{for} $

$F^{656}_{6}  $

\item $\left(
\begin{array}{cccc}
 a & \bar{a} & \frac{1}{2} i \sqrt{-1+\sqrt{5}} & -\frac{1}{2} i \sqrt{1-\frac{1}{\sqrt{5}}} \\
 \bar{a} & a & -\frac{1}{2} i \sqrt{-1+\sqrt{5}} & -\frac{1}{2} i \sqrt{1-\frac{1}{\sqrt{5}}} \\
 -\frac{1}{2} i \sqrt{-1+\sqrt{5}} & \frac{1}{2} i \sqrt{-1+\sqrt{5}} & \frac{1}{2} i \left(-1+\sqrt{5}\right) & 0 \\
 \frac{1}{2} \sqrt{1-\frac{1}{\sqrt{5}}} & \frac{1}{2} \sqrt{1-\frac{1}{\sqrt{5}}} & 0 & i \sqrt{\frac{1}{10} \left(5+\sqrt{5}\right)} \\
\end{array}
\right)\: \: \textrm{for} $

$F^{665}_{6}  $

\item $\left(
\begin{array}{cccc}
 -\bar{a} & -a & \frac{1}{2} i \sqrt{-1+\sqrt{5}} & -\frac{1}{2} i \sqrt{1-\frac{1}{\sqrt{5}}} \\
 -a & -\bar{a} & -\frac{1}{2} i \sqrt{-1+\sqrt{5}} & -\frac{1}{2} i \sqrt{1-\frac{1}{\sqrt{5}}} \\
 -\frac{1}{2} i \sqrt{-1+\sqrt{5}} & \frac{1}{2} i \sqrt{-1+\sqrt{5}} & \frac{1}{2} i \left(-1+\sqrt{5}\right) & 0 \\
 \frac{1}{2} i \sqrt{1-\frac{1}{\sqrt{5}}} & \frac{1}{2} i \sqrt{1-\frac{1}{\sqrt{5}}} & 0 & \sqrt{\frac{1}{10} \left(5+\sqrt{5}\right)} \\
\end{array}
\right)\: \: \textrm{for} $

$F^{666}_{5}  $

\item $\left(
\begin{array}{ccccc}
 \sqrt{-2+\sqrt{5}} & \frac{1}{2} \left(3-\sqrt{5}\right) & \frac{1}{2} \left(-3+\sqrt{5}\right) & -\sqrt{2 \left(-2+\sqrt{5}\right)} & 0 \\
 \frac{1}{2} \left(1-\sqrt{5}\right) & \frac{1}{2} \sqrt{5 \left(-2+\sqrt{5}\right)} & -\frac{1}{2} \sqrt{5 \left(-2+\sqrt{5}\right)} & \frac{1}{\sqrt{18+8 \sqrt{5}}} & 0 \\
 0 & \frac{1}{2} & \frac{1}{2} & 0 & -\frac{1}{\sqrt{2}} \\
 0 & -\frac{1}{2} & -\frac{1}{2} & 0 & -\frac{1}{\sqrt{2}} \\
 \frac{1}{2} \left(-1+\sqrt{5}\right) & \frac{1}{2} \sqrt{-2+\sqrt{5}} & -\frac{1}{2} \sqrt{-2+\sqrt{5}} & \frac{1}{\sqrt{2}} & 0 \\
\end{array}
\right)\: \: \textrm{for} $

$F^{446}_{6}  $

\item $\left(
\begin{array}{ccccc}
 2-\sqrt{5} & -\sqrt{\frac{1}{2} \left(-15+7 \sqrt{5}\right)} & 0 & 0 & i \sqrt{\frac{1}{2} \left(-1+\sqrt{5}\right)} \\
 \sqrt{\frac{1}{2} \left(-15+7 \sqrt{5}\right)} & -\frac{3}{4} \left(-3+\sqrt{5}\right) & 0 & 0 & \frac{1}{2} i \sqrt{\frac{1}{2} \left(5-\sqrt{5}\right)} \\
 0 & 0 & \frac{1}{2} i \sqrt{\frac{1}{2} \left(5-\sqrt{5}\right)} & \frac{1}{4} \left(-1-\sqrt{5}\right) & 0 \\
 0 & 0 & \frac{1}{4} \left(-1-\sqrt{5}\right) & \frac{1}{2} i \sqrt{\frac{1}{2} \left(5-\sqrt{5}\right)} & 0 \\
 -i \sqrt{\frac{1}{2} \left(-1+\sqrt{5}\right)} & \frac{1}{2} i \sqrt{\frac{1}{2} \left(5-\sqrt{5}\right)} & 0 & 0 & \frac{1}{4} \left(-3+\sqrt{5}\right) \\
\end{array}
\right)\: \: \textrm{for} $

$F^{464}_{6}  $

\item $\left(
\begin{array}{ccccc}
 \sqrt{-2+\sqrt{5}} & \frac{1}{2} \left(-1+\sqrt{5}\right) & 0 & 0 & -\frac{1}{2} i \left(-1+\sqrt{5}\right) \\
 \sqrt{\frac{1}{2} \left(7-3 \sqrt{5}\right)} & -\frac{1}{2} \sqrt{5 \left(-2+\sqrt{5}\right)} & -\frac{i}{2} & \frac{1}{2} & -\frac{1}{2} i \sqrt{-2+\sqrt{5}} \\
 -\sqrt{\frac{1}{2} \left(7-3 \sqrt{5}\right)} & \frac{1}{2} \sqrt{5 \left(-2+\sqrt{5}\right)} & -\frac{i}{2} & \frac{1}{2} & \frac{1}{2} i \sqrt{-2+\sqrt{5}} \\
 \sqrt{2 \left(-2+\sqrt{5}\right)} & \frac{1}{\sqrt{18+8 \sqrt{5}}} & 0 & 0 & \frac{i}{\sqrt{2}} \\
 0 & 0 & -\frac{1}{\sqrt{2}} & \frac{i}{\sqrt{2}} & 0 \\
\end{array}
\right)\: \: \textrm{for} $

$F^{466}_{4}  $

\item $\left(
\begin{array}{ccccc}
 \sqrt{-2+\sqrt{5}} & \frac{1}{2} \left(-1+\sqrt{5}\right) & 0 & 0 & \frac{1}{2} \left(1-\sqrt{5}\right) \\
 -\frac{1}{2} i \left(-3+\sqrt{5}\right) & -\frac{1}{2} i \sqrt{5 \left(-2+\sqrt{5}\right)} & -\frac{i}{2} & \frac{i}{2} & -\frac{1}{2} i \sqrt{-2+\sqrt{5}} \\
 \frac{1}{2} i \left(-3+\sqrt{5}\right) & \frac{1}{2} i \sqrt{5 \left(-2+\sqrt{5}\right)} & -\frac{i}{2} & \frac{i}{2} & \frac{1}{2} i \sqrt{-2+\sqrt{5}} \\
 i \sqrt{2 \left(-2+\sqrt{5}\right)} & \frac{i}{\sqrt{18+8 \sqrt{5}}} & 0 & 0 & \frac{i}{\sqrt{2}} \\
 0 & 0 & -\frac{i}{\sqrt{2}} & -\frac{i}{\sqrt{2}} & 0 \\
\end{array}
\right)\: \: \textrm{for} $

$F^{644}_{6}  $

\item $\left(
\begin{array}{ccccc}
 -i \left(-2+\sqrt{5}\right) & -i \sqrt{\frac{1}{2} \left(-15+7 \sqrt{5}\right)} & 0 & 0 & i \sqrt{\frac{1}{2} \left(-1+\sqrt{5}\right)} \\
 i \sqrt{\frac{1}{2} \left(-15+7 \sqrt{5}\right)} & -\frac{3}{4} i \left(-3+\sqrt{5}\right) & 0 & 0 & \frac{1}{2} i \sqrt{\frac{1}{2} \left(5-\sqrt{5}\right)} \\
 0 & 0 & -\frac{1}{4} i \left(1+\sqrt{5}\right) & -\frac{1}{2} i \sqrt{\frac{1}{2} \left(5-\sqrt{5}\right)} & 0 \\
 0 & 0 & -\frac{1}{2} i \sqrt{\frac{1}{2} \left(5-\sqrt{5}\right)} & \frac{1}{4} i \left(1+\sqrt{5}\right) & 0 \\
 -i \sqrt{\frac{1}{2} \left(-1+\sqrt{5}\right)} & \frac{1}{2} i \sqrt{\frac{1}{2} \left(5-\sqrt{5}\right)} & 0 & 0 & -\frac{1}{4} i \left(-3+\sqrt{5}\right) \\
\end{array}
\right)\: \: \textrm{for} $

$F^{646}_{4}  $

\item $\left(
\begin{array}{ccccc}
 i \sqrt{-2+\sqrt{5}} & -i \sqrt{\frac{1}{2} \left(7-3 \sqrt{5}\right)} & i \sqrt{\frac{1}{2} \left(7-3 \sqrt{5}\right)} & i \sqrt{2 \left(-2+\sqrt{5}\right)} & 0 \\
 -\frac{1}{2} i \left(-1+\sqrt{5}\right) & -\frac{1}{2} i \sqrt{5 \left(-2+\sqrt{5}\right)} & \frac{1}{2} i \sqrt{5 \left(-2+\sqrt{5}\right)} & -\frac{i}{\sqrt{18+8 \sqrt{5}}} & 0 \\
 0 & \frac{1}{2} & \frac{1}{2} & 0 & \frac{i}{\sqrt{2}} \\
 0 & \frac{i}{2} & \frac{i}{2} & 0 & \frac{1}{\sqrt{2}} \\
 \frac{1}{2} \left(1-\sqrt{5}\right) & \frac{1}{2} \sqrt{-2+\sqrt{5}} & -\frac{1}{2} \sqrt{-2+\sqrt{5}} & \frac{1}{\sqrt{2}} & 0 \\
\end{array}
\right)\: \: \textrm{for} $

$F^{664}_{4}  $

\item $\left(
\begin{array}{ccccccc}
 -\sqrt{1-\frac{2}{\sqrt{5}}} & -\sqrt{\frac{1}{10} \left(-5+3 \sqrt{5}\right)} & \sqrt{\frac{1}{10} \left(-5+3 \sqrt{5}\right)} & -\sqrt{\frac{1}{10} \left(5-\sqrt{5}\right)} & 0 & 0 & i \sqrt{\frac{1}{10} \left(5-\sqrt{5}\right)} \\
 -i \sqrt{\frac{1}{10} \left(-5+3 \sqrt{5}\right)} & i \sqrt{1-\frac{2}{\sqrt{5}}} & -i \sqrt{1-\frac{2}{\sqrt{5}}} & \frac{i}{2 \sqrt{20+9 \sqrt{5}}} & -\frac{1}{2} & -\frac{i}{2} & -\frac{1}{2 \sqrt[4]{5}} \\
 i \sqrt{\frac{1}{10} \left(-5+3 \sqrt{5}\right)} & -i \sqrt{1-\frac{2}{\sqrt{5}}} & i \sqrt{1-\frac{2}{\sqrt{5}}} & -\frac{i}{2 \sqrt{20+9 \sqrt{5}}} & -\frac{1}{2} & -\frac{i}{2} & \frac{1}{2 \sqrt[4]{5}} \\
 i \sqrt{\frac{1}{10} \left(5-\sqrt{5}\right)} & -\frac{i}{2 \sqrt{20+9 \sqrt{5}}} & \frac{i}{2 \sqrt{20+9 \sqrt{5}}} & \frac{3}{2} i \sqrt{1-\frac{2}{\sqrt{5}}} & 0 & 0 & -\frac{1}{2} \sqrt{1+\frac{2}{\sqrt{5}}} \\
 0 & \frac{1}{2} & \frac{1}{2} & 0 & \frac{1}{2} & -\frac{i}{2} & 0 \\
 0 & \frac{i}{2} & \frac{i}{2} & 0 & -\frac{i}{2} & -\frac{1}{2} & 0 \\
 \sqrt{\frac{1}{10} \left(5-\sqrt{5}\right)} & \frac{1}{2 \sqrt[4]{5}} & -\frac{1}{2 \sqrt[4]{5}} & -\frac{1}{2} \sqrt{1+\frac{2}{\sqrt{5}}} & 0 & 0 & \frac{1}{2} i \sqrt{1-\frac{2}{\sqrt{5}}} \\
\end{array}
\right)\: \: \textrm{for} $

$F^{666}_{6}  $
}
}
\end{itemize}

\subsection{$R$-matrices}

The $R$-matrices are defined by the following figure, where $R_{c;ij}^{\leftidx{^a}{b}{}a}$ is the $(i,j)$-entry of the matrix $R_{c}^{\leftidx{^a}{b}{}a}$, and $\leftidx{^a}{b}{}$ means the action of $h$ on $b$, where $a \in \C_h$. In the list below, when the dimension of $R_{c}^{ab}$ is $1$, we will just write it as a number like the usual convention.

\setlength{\unitlength}{0.030in}
\begin{picture}(50,60)(0,0)
 \put(10,10){\line(0,-1){10}}
 \put(10,10){\line(1,1){10}}
 \put(10,10){\line(-1,1){10}}
 \put(0,20){\line(1,1){20}}
 \put(20,20){\line(-1,1){8}}
 \put(8,32){\line(-1,1){8}}

 \put(2,20){$a$}
 \put(22,20){$b$}
 \put(12,2){$c$}
 \put(2,40){$\leftidx{^a}{b}{}$}
 \put(22,40){$a$}
 \put(5,8){$j$}

 \put(25,20){$ = \sum\limits_{i}$}
 \put(38,20){$R_{c;ij}^{\leftidx{^a}{b}{}a}$}

 \put(70,10){\line(0,-1){10}}
 \put(70,10){\line(1,1){30}}
 \put(70,10){\line(-1,1){30}}

 \put(42,40){$\leftidx{^a}{b}{}$}
 \put(102,40){$a$}
 \put(72,2){$c$}
 \put(65,8){$i$}
\end{picture}

\begin{itemize}
\item $1\: \: \textrm{ for } \: \: R^{11}_{1}, R^{12}_{2}, R^{13}_{3}, R^{14}_{4}, R^{15}_{5}, R^{16}_{6}, R^{21}_{2}, R^{23}_{4}, R^{31}_{3}, R^{32}_{4}, R^{41}_{4}, R^{46}_{5}, R^{51}_{5}, R^{61}_{6}$

\item $e^{\frac{4 i \pi }{5}}\: \: \textrm{for} \: \: R^{22}_{1}, R^{33}_{1}, R^{44}_{4}$

\item $e^{-\frac{3 i \pi }{5}}\: \: \textrm{for} \: \: R^{22}_{2}, R^{33}_{3}$

\item $e^{-\frac{i \pi }{5}}\: \: \textrm{for} \: \: R^{24}_{3}, R^{34}_{2}, R^{42}_{3}, R^{43}_{2}$

\item $e^{\frac{2 i \pi }{5}}\: \: \textrm{for} \: \: R^{24}_{4}, R^{34}_{4}, R^{42}_{4}, R^{43}_{4}$

\item $e^{\frac{i \pi }{5}}\: \: \textrm{for} \: \: R^{25}_{6}, R^{35}_{6}, R^{44}_{2}, R^{44}_{3}, R^{52}_{6}, R^{53}_{6}$

\item $e^{\frac{3 i \pi }{5}}\: \: \textrm{for} \: \: R^{26}_{5}, R^{36}_{5}, R^{54}_{6}, R^{62}_{5}, R^{63}_{5}$

\item $e^{-\frac{i \pi }{10}}\: \: \textrm{for} \: \: R^{26}_{6}, R^{36}_{6}, R^{62}_{6}, R^{63}_{6}$

\item $e^{-\frac{2 i \pi }{5}}\: \: \textrm{for} \: \: R^{44}_{1}, R^{45}_{6}$

\item $e^{\frac{3 i \pi }{10}}\: \: \textrm{for} \: \: R^{45}_{5}$

\item $e^{-\frac{7 i \pi }{10}}\: \: \textrm{for} \: \: R^{54}_{5}$

\item $e^{\frac{17 i \pi }{20}}\: \: \textrm{for} \: \: R^{55}_{1}, R^{66}_{2}, R^{66}_{3}$

\item $e^{\frac{i \pi }{20}}\: \: \textrm{for} \: \: R^{55}_{4}$

\item $e^{-\frac{17 i \pi }{20}}\: \: \textrm{for} \: \: R^{56}_{2}, R^{56}_{3}, R^{65}_{2}, R^{65}_{3}$

\item $e^{-\frac{i \pi }{4}}\: \: \textrm{for} \: \: R^{56}_{4}, R^{65}_{4}$

\item $-1\: \: \textrm{for} \: \: R^{64}_{5}$

\item $e^{\frac{i \pi }{4}}\: \: \textrm{for} \: \: R^{66}_{1}$

\item $\left(
\begin{array}{cc}
 \frac{1}{4} \left(-i-i \sqrt{5}-\sqrt{2 \left(5-\sqrt{5}\right)}\right) & 0 \\
 0 & \frac{1}{4} i \left(1+\sqrt{5}-i \sqrt{2 \left(5-\sqrt{5}\right)}\right) \\
\end{array}
\right)\: \: \textrm{for} \: \: R^{46}_{6}$

\item $\left(
\begin{array}{cc}
 \frac{1}{4} \left(i+i \sqrt{5}+\sqrt{2 \left(5-\sqrt{5}\right)}\right) & 0 \\
 0 & \frac{1}{4} \left(i+i \sqrt{5}+\sqrt{2 \left(5-\sqrt{5}\right)}\right) \\
\end{array}
\right)\: \: \textrm{for} \: \: R^{64}_{6}$

\item $\left(
\begin{array}{cc}
 \left(-\frac{1}{8}-\frac{i}{8}\right) \left(\sqrt{2}+\sqrt{10}+2 i \sqrt{5-\sqrt{5}}\right) & 0 \\
 0 & \left(\frac{1}{8}+\frac{i}{8}\right) \left(\sqrt{2}+\sqrt{10}+2 i \sqrt{5-\sqrt{5}}\right) \\
\end{array}
\right)\: \: \textrm{for} \: \: R^{66}_{4}$
\end{itemize}

\subsection{$U$ symbols and $\eta$ symbols}
The $U$-symbols are defined as follows, where $U_h(a,b;c)_{ij}$ is the $(i,j)$-entry of $U_h(a,b;c)$, and $h$ is a group element. When $h=e$, the $U$-symbols are always identity matrices, so we omit them below in the list. The $\eta$ symbols are defined by the isomorphism:
$$\eta_a(h,k): \: \leftidx{^{(hk)}}{a}{} \longrightarrow \leftidx{^h}{(\leftidx{^k}{a}{})}{}$$

For the category $\C_{\Z_2}^{\times}$, all the $\eta$ symbols are equal to $1$.

\setlength{\unitlength}{0.030in}
\begin{picture}(100,40)(-10,0)
\put(-5,10){$\rho_h($}

 \put(10,10){\line(0,-1){10}}
 \put(10,10){\line(1,1){10}}
 \put(10,10){\line(-1,1){10}}

 \put(2,20){$a$}
 \put(18,20){$b$}
 \put(12,2){$c$}
 \put(5,8){$j$}

 \put(22,10){$)$}

 \put(25,10){$ = \sum\limits_{i}$}
 \put(38,10){$U_h(a,b;c)_{ij}$}

 \put(70,10){\line(0,-1){10}}
 \put(70,10){\line(1,1){10}}
 \put(70,10){\line(-1,1){10}}

 \put(62,20){$\leftidx{^h}{a}{}$}
 \put(82,20){$\leftidx{^h}{b}{}$}
 \put(72,2){$\leftidx{^h}{c}{}$}
 \put(65,8){$i$}
\end{picture}

\begin{itemize}
\item $1\: \: \textrm{for} \: \: \\
 U_g(1,1;1), U_g(1,2;2), U_g(1,3;3), U_g(1,4;4), U_g(1,5;5), \\
 U_g(1,6;6), U_g(2,1;2), U_g(2,2;1), U_g(2,2;2), U_g(2,3;4), \\
 U_g(3,1;3), U_g(3,2;4), U_g(3,3;1), U_g(3,3;3), U_g(4,1;4),\\
 U_g(4,4;1), U_g(4,4;2), U_g(4,4;3), U_g(4,4;4), U_g(4,5;5), \\
 U_g(4,5;6), U_g(4,6;5), U_g(5,1;5), U_g(5,4;5), U_g(5,4;6), \\
 U_g(5,5;1), U_g(5,5;4), U_g(5,6;2), U_g(5,6;3), U_g(5,6;4),\\
 U_g(6,1;6), U_g(6,4;5), U_g(6,5;2), U_g(6,5;3), U_g(6,5;4), \\
 U_g(6,6;1), U_g(6,6;2), U_g(6,6;3)$ \\

\item $-1\: \: \textrm{for} \: \:  \\
U_g(2,4;3), U_g(2,4;4), U_g(2,5;6), U_g(2,6;5), U_g(2,6;6),\\
U_g(3,4;2), U_g(3,4;4), U_g(3,5;6), U_g(3,6;5), U_g(3,6;6), \\
U_g(4,2;3), U_g(4,2;4), U_g(4,3;2), U_g(4,3;4), U_g(5,2;6), \\
U_g(5,3;6), U_g(6,2;5), U_g(6,2;6), U_g(6,3;5), U_g(6,3;6)$\\

\item $\left(
\begin{array}{cc}
 1 & 0 \\
 0 & -1 \\
\end{array}
\right)\: \: \textrm{for} \: \: U_g(4,6;6), U_g(6,4;6), U_g(6,6;4)$
\end{itemize}

\section{Data for $\gauge{(\SO(8)_1)}{\Sym_3}{\Sym_3}$}
\label{sec:dataSO8}

The two theories resulting from gauging $\SO(8)_1$ share some common fusion rules and also have their own fusion rules. 

The common fusion rules of the two theories:

\begin{itemize}
\item $(\extunit,-1) \otimes (\extunit,-1) = \unit $

\item $(\extunit,-1) \otimes a = a $

\item $(\extunit,-1) \otimes (Y,1) = (Y,-1) $

\item $(\extunit,-1) \otimes (Y,-1) = (Y,1) $

\item $(\extunit,-1) \otimes X = X $

\item $(\extunit,-1) \otimes aX = aX $

\item $(\extunit,-1) \otimes aX' = aX' $

\item $(\extunit,-1) \otimes (X^{+},1) = (X^{+},-1) $

\item $(\extunit,-1) \otimes (X^{+},-1) = (X^{+},1) $

\item $(\extunit,-1) \otimes (X^{-},1) = (X^{-},-1) $

\item $(\extunit,-1) \otimes (X^{-},-1) = (X^{-},1) $

\item $a \otimes a = \unit  \oplus (\extunit,-1) \oplus a$

\item $a \otimes (Y,1) = (Y,1)  \oplus (Y,-1)$

\item $a \otimes (Y,-1) = (Y,1)  \oplus (Y,-1)$

\item $a \otimes X = aX  \oplus aX'$

\item $a \otimes aX = X  \oplus aX'$

\item $a \otimes aX' = X  \oplus aX$

\item $a \otimes (X^{+},1) = (X^{+},1)  \oplus (X^{+},-1)$

\item $a \otimes (X^{+},-1) = (X^{+},1)  \oplus (X^{+},-1)$

\item $a \otimes (X^{-},1) = (X^{-},1)  \oplus (X^{-},-1)$

\item $a \otimes (X^{-},-1) = (X^{-},1)  \oplus (X^{-},-1)$

\item $(Y,1) \otimes (Y,1) = \unit  \oplus a \oplus (Y,1) \oplus (Y,-1)$

\item $(Y,1) \otimes (Y,-1) = (\extunit,-1)  \oplus a \oplus (Y,1) \oplus (Y,-1)$

\item $(Y,1) \otimes X = X  \oplus aX \oplus aX'$

\item $(Y,1) \otimes aX = X  \oplus aX \oplus aX'$

\item $(Y,1) \otimes aX' = X  \oplus aX \oplus aX'$

\item $(Y,-1) \otimes (Y,-1) = \unit  \oplus a \oplus (Y,1) \oplus (Y,-1)$

\item $(Y,-1) \otimes X = X  \oplus aX \oplus aX'$

\item $(Y,-1) \otimes aX = X  \oplus aX \oplus aX'$

\item $(Y,-1) \otimes aX' = X  \oplus aX \oplus aX'$

\item $X \otimes X = \unit  \oplus (\extunit,-1) \oplus (Y,1) \oplus (Y,-1) \oplus X \oplus aX'$

\item $X \otimes aX = a  \oplus (Y,1) \oplus (Y,-1) \oplus aX \oplus aX'$

\item $X \otimes aX' = a  \oplus (Y,1) \oplus (Y,-1) \oplus X \oplus aX$

\item $X \otimes (X^{+},1) = (X^{+},1)  \oplus (X^{+},-1) \oplus (X^{-},1) \oplus (X^{-},-1)$

\item $X \otimes (X^{+},-1) = (X^{+},1)  \oplus (X^{+},-1) \oplus (X^{-},1) \oplus (X^{-},-1)$

\item $X \otimes (X^{-},1) = (X^{+},1)  \oplus (X^{+},-1) \oplus (X^{-},1) \oplus (X^{-},-1)$

\item $X \otimes (X^{-},-1) = (X^{+},1)  \oplus (X^{+},-1) \oplus (X^{-},1) \oplus (X^{-},-1)$

\item $aX \otimes aX = \unit  \oplus (\extunit,-1) \oplus (Y,1) \oplus (Y,-1) \oplus X \oplus aX$

\item $aX \otimes aX' = a  \oplus (Y,1) \oplus (Y,-1) \oplus X \oplus aX'$

\item $aX \otimes (X^{+},1) = (X^{+},1)  \oplus (X^{+},-1) \oplus (X^{-},1) \oplus (X^{-},-1)$

\item $aX \otimes (X^{+},-1) = (X^{+},1)  \oplus (X^{+},-1) \oplus (X^{-},1) \oplus (X^{-},-1)$

\item $aX \otimes (X^{-},1) = (X^{+},1)  \oplus (X^{+},-1) \oplus (X^{-},1) \oplus (X^{-},-1)$

\item $aX \otimes (X^{-},-1) = (X^{+},1)  \oplus (X^{+},-1) \oplus (X^{-},1) \oplus (X^{-},-1)$

\item $aX' \otimes aX' = \unit  \oplus (\extunit,-1) \oplus (Y,1) \oplus (Y,-1) \oplus aX \oplus aX'$

\item $aX' \otimes (X^{+},1) = (X^{+},1)  \oplus (X^{+},-1) \oplus (X^{-},1) \oplus (X^{-},-1)$

\item $aX' \otimes (X^{+},-1) = (X^{+},1)  \oplus (X^{+},-1) \oplus (X^{-},1) \oplus (X^{-},-1)$

\item $aX' \otimes (X^{-},1) = (X^{+},1)  \oplus (X^{+},-1) \oplus (X^{-},1) \oplus (X^{-},-1)$

\item $aX' \otimes (X^{-},-1) = (X^{+},1)  \oplus (X^{+},-1) \oplus (X^{-},1) \oplus (X^{-},-1)$

\item $(X^{+},1) \otimes (X^{-},1) = (Y,1)  \oplus (Y,-1) \oplus X \oplus aX \oplus aX'$

\item $(X^{+},1) \otimes (X^{-},-1) = (Y,1)  \oplus (Y,-1) \oplus X \oplus aX \oplus aX'$

\item $(X^{+},-1) \otimes (X^{-},1) = (Y,1)  \oplus (Y,-1) \oplus X \oplus aX \oplus aX'$

\item $(X^{+},-1) \otimes (X^{-},-1) = (Y,1)  \oplus (Y,-1) \oplus X \oplus aX \oplus aX'$
\end{itemize}

Fusion rules for Theory 1:
\begin{itemize}
\item $(Y,1) \otimes (X^{+},1) = (X^{+},1)  \oplus (X^{-},1) \oplus (X^{-},-1)$

\item $(Y,1) \otimes (X^{+},-1) = (X^{+},-1)  \oplus (X^{-},1) \oplus (X^{-},-1)$

\item $(Y,1) \otimes (X^{-},1) = (X^{+},1)  \oplus (X^{+},-1) \oplus (X^{-},-1)$

\item $(Y,1) \otimes (X^{-},-1) = (X^{+},1)  \oplus (X^{+},-1) \oplus (X^{-},1)$

\item $(Y,-1) \otimes (X^{+},1) = (X^{+},-1)  \oplus (X^{-},1) \oplus (X^{-},-1)$

\item $(Y,-1) \otimes (X^{+},-1) = (X^{+},1)  \oplus (X^{-},1) \oplus (X^{-},-1)$

\item $(Y,-1) \otimes (X^{-},1) = (X^{+},1)  \oplus (X^{+},-1) \oplus (X^{-},1)$

\item $(Y,-1) \otimes (X^{-},-1) = (X^{+},1)  \oplus (X^{+},-1) \oplus (X^{-},-1)$

\item $(X^{+},1) \otimes (X^{+},1) = \unit  \oplus a \oplus (Y,1) \oplus X \oplus aX \oplus aX'$

\item $(X^{+},1) \otimes (X^{+},-1) = (\extunit,-1)  \oplus a \oplus (Y,-1) \oplus X \oplus aX \oplus aX'$

\item $(X^{+},-1) \otimes (X^{+},-1) = \unit  \oplus a \oplus (Y,1) \oplus X \oplus aX \oplus aX'$

\item $(X^{-},1) \otimes (X^{-},1) = \unit  \oplus a \oplus (Y,-1) \oplus X \oplus aX \oplus aX'$

\item $(X^{-},1) \otimes (X^{-},-1) = (\extunit,-1)  \oplus a \oplus (Y,1) \oplus X \oplus aX \oplus aX'$

\item $(X^{-},-1) \otimes (X^{-},-1) = \unit  \oplus a \oplus (Y,-1) \oplus X \oplus aX \oplus aX'$
\end{itemize}

Fusion rules for Theory 2:
\begin{itemize}
\item $(Y,1) \otimes (X^{+},1) = (X^{+},-1)  \oplus (X^{-},1) \oplus (X^{-},-1)$

\item $(Y,1) \otimes (X^{+},-1) = (X^{+},1)  \oplus (X^{-},1) \oplus (X^{-},-1)$

\item $(Y,1) \otimes (X^{-},1) = (X^{+},1)  \oplus (X^{+},-1) \oplus (X^{-},1)$

\item $(Y,1) \otimes (X^{-},-1) = (X^{+},1)  \oplus (X^{+},-1) \oplus (X^{-},-1)$

\item $(Y,-1) \otimes (X^{+},1) = (X^{+},1)  \oplus (X^{-},1) \oplus (X^{-},-1)$

\item $(Y,-1) \otimes (X^{+},-1) = (X^{+},-1)  \oplus (X^{-},1) \oplus (X^{-},-1)$

\item $(Y,-1) \otimes (X^{-},1) = (X^{+},1)  \oplus (X^{+},-1) \oplus (X^{-},-1)$

\item $(Y,-1) \otimes (X^{-},-1) = (X^{+},1)  \oplus (X^{+},-1) \oplus (X^{-},1)$

\item $(X^{+},1) \otimes (X^{+},1) = \unit  \oplus a \oplus (Y,-1) \oplus X \oplus aX \oplus aX'$

\item $(X^{+},1) \otimes (X^{+},-1) = (\extunit,-1)  \oplus a \oplus (Y,1) \oplus X \oplus aX \oplus aX'$

\item $(X^{+},-1) \otimes (X^{+},-1) = \unit  \oplus a \oplus (Y,-1) \oplus X \oplus aX \oplus aX'$

\item $(X^{-},1) \otimes (X^{-},1) = \unit  \oplus a \oplus (Y,1) \oplus X \oplus aX \oplus aX'$

\item $(X^{-},1) \otimes (X^{-},-1) = (\extunit,-1)  \oplus a \oplus (Y,-1) \oplus X \oplus aX \oplus aX'$

\item $(X^{-},-1) \otimes (X^{-},-1) = \unit  \oplus a \oplus (Y,1) \oplus X \oplus aX \oplus aX'$
\end{itemize}

%% file: bibliogauging.tex
\bibliographystyle{amsalpha}